\documentclass[a4paper,reqno]{amsart}
\usepackage{amssymb}
\usepackage{latexsym}
\usepackage{amsmath}
\usepackage{euscript}
\usepackage{graphics,color}
\usepackage[all]{xy}
\usepackage[margin=3cm]{geometry}
\usepackage{MnSymbol} 
\usepackage{tikz}
\usepackage{tikz-3dplot}
\usetikzlibrary{calc}
\usepackage[normalem]{ulem}
\usepackage{enumitem}
\usepackage{caption}

\newcommand{\be}{\begin{equation}}
\newcommand{\ee}{\end{equation}}
\newcommand{\ba}{\begin{eqnarray}}
\newcommand{\ea}{\end{eqnarray}}
\newcommand{\baa}{\begin{eqnarray*}}
\newcommand{\eaa}{\end{eqnarray*}}
\newcommand{\bb}{\begin{thebibliography}}
\newcommand{\eb}{\end{thebibliography}}

\newcounter{my}
\newcommand{\he}%
   {\stepcounter{equation}\setcounter{my}%
   {\value{equation}}\setcounter{equation}0%
   }%
\newcommand{\she}%
   {\setcounter{equation}{\value{my}}%
    }%

\newtheorem{pr}{Proposition}
\newtheorem{cor}{Corollary}

\newtheorem{theorem}{Theorem}[section]

\newtheorem{lemma}[theorem]{Lemma}

\theoremstyle{definition}

\newtheorem{remark}[theorem]{Remark}

\numberwithin{equation}{section}

\newcommand{\hg}[2]{\,\mbox{}_{#1}F_{ #2}\!}

\newcommand{\argu}[3]{\left(\begin{array}{c} #1\\#2\end{array} ; #3\right)}

\usepackage{tikz-3dplot}
\DeclareRobustCommand{\erase}{\bgroup\markoverwith{\textcolor{red}{\rule[.5ex]{2pt}{0.8pt}}}\ULon}

\title[Algebraic interpretation of the two-variable Jacobi polynomials ]
  {Algebraic interpretation of the two-variable Jacobi polynomials on the triangle: the pentagonal way}

\author[Crampé, Labriet, Morey, Tsujimoto, Vinet, Zhedanov]
  {Nicolas Crampé, Quentin Labriet, Lucia Morey, Satoshi Tsujimoto, Luc Vinet and Alexei Zhedanov}

\address{CNRS - Université de Montréal CRM-CNRS, P.O. Box 6128, Centre-ville Station, Montr\'eal (Qu\'ebec), H3C 3J7}

\address{Centre de recherches math\'ematiques, Universit\'e de Montr\'eal, P.O. Box 6128, Centre-ville Station, Montr\'eal (Qu\'ebec), H3C 3J7}

\address{Centre de recherches math\'ematiques, Universit\'e de Montr\'eal, P.O. Box 6128, Centre-ville Station, Montr\'eal (Qu\'ebec), H3C 3J7}

\address{Graduate School of Informatics, Kyoto University,
Yoshida-Honmachi, Kyoto, Japan 606-8501}

\address{IVADO and Centre de recherches math\'ematiques, Universit\'e de Montr\'eal, P.O. Box 6128, Centre-ville Station, Montr\'eal (Qu\'ebec), H3C 3J7}

\address{Department of Mathematics, School of Information, Renmin University of China, Beijing 100872,CHINA}

\begin{document}

\begin{abstract}
The rank two Jacobi algebra $\mathcal{J}_2$ is used to provide an interpretation of the two-variable Jacobi polynomials $J_{n,k}^{(a,b,c)}(x,y)$ on the triangle, as overlaps between two representation bases. The subalgebra structure of $\mathcal{J}_2$ depicted via a pentagonal graph is exploited to find the explicit expression of the two-variable functions in terms of univariate Jacobi polynomials. It is also seen to provide an explanation for the fact that the expansion on the basis $J_{n,k}^{(a,b,c)}(x,y)$ of the polynomials obtained from the latter by permuting the variables $x,y, z=1-x-y$ and the parameters $(a,b,c)$ is given in terms of Racah polynomials. The underlying order-three symmetry is discussed.
\end{abstract}

\maketitle

\section{Introduction}
The two-variable Jacobi polynomials on the triangle \cite{dunkl2014orthogonal} have been initially introduced by Proriol \cite{proriol1957famille} and brought to the fore in particular by Koornwinder in his studies of two-variable orthogonal polynomials \cite{koornwinder1975two}. They are of total degree $n+k$ in the variables $x$ and $y$ restricted to $0\le x \le 1-y \le1$ and read as follows:
\begin{align}
 &J_{n,k}^{(a,b,c)}(x,y)=J_{n-k}^{(a,b+c+2k+1)}\bigl(x\bigr) \ (1-x)^k \ J_k^{(b,c)}\left(\frac{y}{1-x}\right), \quad \mbox{}\quad n \ge k \ge 0, \label{defJac2}
\end{align}
in terms of the univariate Jacobi polynomial $J_n^{(a,b)}(x)$ on the interval $[0,1]$ given by
\begin{align}
  &  J_n^{(a,b)}(x) = \dfrac{(a+1)_n}{n!}\ \hg{2}{1}\argu{-n,n+a+b+1}{a+1}{x}. \label{defJac1}
\end{align}
The univariate polynomials $J_n^{(a,b)}(x)$ obey the following orthogonality relation on the interval $[0,1]$:
\begin{align}
&   \int_{0}^1 J_n^{(a,b)}(x) J_{m}^{(a,b)}(x)\ x^{a} (1-x)^{b}dx
     =  N_n^{(a,b)}\delta_{nm}, \quad a,b>-1,\label{orth1var}
\end{align}
where
\begin{equation}
    N_n^{(a,b)}=\frac{\Gamma(n+a+1)\Gamma(n+b+1)}{(2n +a+b+1)\,n!\, \Gamma(n+a+b+1) }. \label{norm1}
\end{equation}
In view of \eqref{orth1var}, the two-variable polynomials $J_{n,k}^{(a,b,c)}(x,y)$ are found to be themselves orthogonal for $a,b,c > -1$ as per:
\begin{align}
 \int_{0\le x\le  1-y \le 1} J_{n,k}^{(a,b,c)}(x,y) J_{n',k'}^{(a,b,c)}(x,y)\ x^a y^b (1-x-y)^c dx dy = N_{n-k}^{(a,b+c+2k+1)} N_k^{(b,c)} \delta_{n,n'} \delta_{k,k'}. 
\end{align}
Apart from their intrinsic interest, the two-variable Jacobi polynomials have found various applications. When the parameters are zero, they have been referred to as Dubiner polynomials \cite{dubiner1991spectral} and have seen much use in finite elements methods applied to fluid dynamics \cite{karniadakis2005spectral}. They have also appeared in the classification by Krall and Sheffer \cite{krall1967orthogonal} and by Engelis \cite{engelis1974some} as well, of the bivariate polynomials obeying a second order partial differential equation; this classification was subsequently connected to that of superintegrable models \cite{vinet2003two} in two dimensions and in fact, the two-variable Jacobi polynomials arise as the wave functions of the generic superintegrable system on the two-sphere \cite{iliev2018symmetry}. 

It is appreciated that the univariate families of hypergeometric polynomials of the Askey scheme admit representation theoretic interpretations based on the algebras that encode their bispectral properties \cite{zhedanov1991hidden}. In the case of the Racah polynomials, this is explained in details in \cite{genest2014superintegrability} or in the review \cite{genest2014racah}. The approach has proved quite fruitful with deep connections to algebraic combinatorics \cite{terwilliger2001two} in particular. 

Higher rank generalizations that apply naturally to multivariate polynomials of discrete variables have been defined by extending the tensorial underpinning of the rank one case. The Racah algebra of arbitrary rank has been introduced in this fashion \cite{de2017higher} and the representations in the rank two case have been worked out with minutia in \cite{crampe2023representations}. While some algebraic results bearing on continuous polynomials (and in particular the bivariate Jacobi polynomials) could be obtained through limits (see for instance, \cite{geronimo2010bispectrality} and \cite{dunkl1984orthogonal} in connection with \cite{dunkl1981difference}), there is still much to understand on that topic.

It is with this motivation, that some of us recently undertook to propose an explicit definition of the rank two Jacobi algebra $\mathcal{J}_2$ \footnote{For historical reasons it has been customary in the literature to denote the rank $n$ algebras by the number $n+2$; we will not follow this convention here and simply label these algebras by $n$.} through a model derived directly from the bispectral properties of the two-variable Jacobi polynomials \cite{crampe2025two}. The goal here is to proceed in reverse and to show how the bivariate Jacobi polynomials can be characterized from the knowledge of the algebra and its representations. Note that the relations that define $\mathcal{J}_2$ have been relegated to appendix \ref{APP A} to ensure that the paper flows smoothly. 

Here is the outline. 
The rank one Jacobi and Racah algebras are reviewed in Section \ref{sec:rank1} with a reminder of how the eponymous polynomials appear through their representations. This knowledge will play a central role in the remainder of the work. It is included for completeness and could be skipped by a reader familiar with this algebraic interpretation of univariate orthogonal polynomials. In Section \ref{sec:2Jacobi} we focus on the identification and description of various rank one subalgebras of $\mathcal{J}_2$ that are of Jacobi and Racah types. These will form the scaffold of the algebraic description of the bivariate Jacobi polynomials and it will prove enlightening to organize these subalgebras in a pentagonal graph. The representation-theoretic interpretation of the two-variable polynomials $J_{n,k}^{(a,b,c)}(x,y)$ will be found in Section \ref{sec:alge}. The various univariate Jacobi polynomials that arise in the overlaps of vectors associated to different representations bases of each of these rank one algebras will be the building blocks of the construction.  The pentagonal subalgebra structure will naturally lead to another set of bivariate Jacobi polynomials related to $J_{n,k}^{(a,b,c)}(x,y)$ by the simultaneous exchanges $a \leftrightarrow c$ and $x \leftrightarrow 1-x-y$. That the Racah polynomials provide the expansion coefficients of the elements of one set in terms of those of the other will be seen in Section \ref{sec:over} to have a natural explanation in the proposed framework. A third family of two-variable Jacobi polynomials related to the standard ones by the simultaneous permutations $x \leftrightarrow y$ and $a \leftrightarrow b$ will be discussed in Section \ref{sec:symm} so as to bring forward the underlying order three symmetry of the description of these functions defined on the triangle. Concluding remarks will be offered in Section \ref{sec:conc}. As already mentioned, the defining relations will be found in Appendix \ref{APP A} while the relevant representations of $\mathcal{J}_2$ are given in Appendix \ref{sec:rep}.

\section{Some rank one algebras and their connections to the corresponding orthogonal polynomials \label{sec:rank1}}

The algebraic interpretation of the two-variable Jacobi polynomials will be based on the rank two-Jacobi algebra $\mathcal{J}_2$ and will rely on the connections between certain rank one subalgebras of $\mathcal{J}_2$ and univariate polynomials. These rank one structures are the Jacobi and Racah algebras. They will be introduced and the representation-theoretic  description of orthogonal polynomials they entail will be reviewed.

\subsection{The Jacobi algebra}

\subsubsection{Definition}

The Jacobi algebra $\mathcal{J}_1$ of rank one \cite{genest2016tridiagonalization} is generated by the elements $K_1$ and $K_2$ that are subjected to the following relations:
\begin{subequations}\label{rank1}
\begin{align} 
  [[K_1,K_2],K_1]=&\;2\{K_1,K_2\}-2 K_1-(\alpha+\beta)(\alpha+\beta+2)K_2+(\alpha+\beta)(\alpha+1),\\
  [K_2,[K_1,K_2]]=&\;2K_2^2-2K_2,
\end{align}    
\end{subequations}
with $[A,B]=AB-BA$, $\{A,B\}=AB+BA$ and $\alpha,\beta$ some parameters. We may remark that this algebra is isomorphic to the rank one Hahn algebra (see for example \cite{frappat2019higgs}).

\subsubsection{Embedding and representations}

It will prove useful to identify the embedding of $\mathfrak{J}_1$ in the Lie algebra $\mathfrak{su}(1,1)$ whose generators $J_0, J_{\pm}$ obey the commutation relations
\begin{equation}
    [J_0, J_{\pm}] = \pm J_{\pm}, \qquad [J_+,J_-]=-2J_0.
\end{equation}
Consider the discrete series bounded from below where the generators act as follows on the basis of orthonormal vectors $|\beta, k\rangle, k=0,1,\dots$:
\begin{align}
    J_0|\beta,k\rangle &= (k+\frac{1}{2}(\beta + 1))|\beta,k\rangle,\\
     J_+|\beta,k\rangle &= a_k\;|\beta,k+1\rangle,\\
     J_-|\beta,k\rangle &= \frac{k(k+\beta)}{a_{k-1}}\;|\beta,k-1\rangle.
\end{align}
(We are not distinguishing in our notation abstract generators from their representations.) In this representation, the Casimir element  $C=J_0^2-J_+J_--J_0$ of $\mathfrak{su}(1,1)$ takes the particular value $C=\tau (\tau -1)$ with $\tau = \frac{1}{2}(\beta +1)$. 
The coefficients $a_k$ depend a priori on the normalization of the vectors $|\beta, k\rangle$. The standard unitary representation is obtained with $a_k=\sqrt{(k+1)(k+\beta +1)}$ but it will be seen that the quantities of interest do not depend on these parameters $a_k$ that we shall hence not bother to specify.

It is then readily checked from the defining relations of $\mathfrak{su}(1,1)$ that the identification
\begin{equation}
    K_1=-J_0^2-J_--\alpha J_0 +\frac{1}{4}(2\alpha + \beta +1)(\beta+1), \qquad K_2=J_++1, \label{id}
\end{equation}
provides an embedding of $\mathfrak{J}_1$ into the enveloping algebra of $\mathfrak{su}(1,1)$.
As a result of the identification \eqref{id}, one obtains the following two-diagonal representation of the Jacobi algebra $\mathfrak{J}_1$:
\begin{align}
    K_1\;|\beta,k\rangle&=-k(k+\alpha+\beta+1)\;|\beta,k\rangle - \frac{k(k+\beta)}{a_{k-1}}\;|\beta,k-1\rangle,\label{k1}\\
    K_2\;|\beta,k\rangle&=|\beta,k\rangle + a_k \;|\beta,k+1\rangle.\label{k2}
\end{align}
This corresponds to one of the classes of $\mathfrak{J}_1$ representations given in \cite{granovskii1992mutual}.

\subsubsection{Eigenvectors}

The eigenvectors of the generators of the Jacobi algebra over this module are readily obtained. 
\begin{enumerate}[leftmargin=*]
\item Eigenvectors $|\mu_n^{(\alpha,\beta)}\rangle$ of $K_1$.\\
    Consider the equation $K_1 \;|\mu_n^{(\alpha,\beta)}\rangle = \mu_n^{(\alpha,\beta)}|\mu_n^{(\alpha,\beta)}\rangle$ with the eigenvalues, $n=0,1,\dots$:
    \begin{equation}
\mu_n^{(\alpha,\beta)} = -n(n+\alpha +\beta +1).
    \end{equation}
Writing $|\mu_n^{(\alpha,\beta)}\rangle = \sum_{k=0}^\infty A_{n,k} \;|\beta,k\rangle$ and enforcing the eigenvalue equation yields the following recurrence relation for the coefficients $A_{n,k}$:
\begin{equation}
-n(n+\alpha + \beta +1) A_{n,k} = -k(k + \alpha +\beta +1) A_{n,k} - \frac{(k+1)(k+\beta +1)}{a_{k}} A_{n,k+1},
\end{equation}
which is solved to give
\begin{equation}
    |\mu_n^{(\alpha,\beta)}\rangle = A_{n,0} \sum_{k=0}^\infty a_0a_1\dots a_{k-1} (-1)^k \frac{(-n)_k}{k!} \frac{(n+\alpha +\beta +1)_k}{(\beta +1)_k} \;|\beta, k\rangle. \label{eigvecK1}
\end{equation}
The initial value $A_{n,0}$ and hence the norm of $|\mu_n^{(\alpha,\beta)}\rangle$ are left unspecified for the moment.

\item Eigenvectors $|x^*\rangle$ of $K_2^{\top}$.\\
Denote by $K_2^{\top}$ the transpose of $K_2$. Focus now on the eigenvalue equation $K_2^{\top}\;|x^*\rangle=x\;|x^*\rangle$ and let $|x^*\rangle = \sum_{\ell=0}^\infty B^*_{\ell}(x) \;|\beta, \ell \rangle$. The eigenvalue equation translates into the recurrence relation $(x-1)B^*_{\ell}(x)=a_{\ell}B^*_{\ell +1}(x)$ from where we find that $|x^*\rangle$ is given by:
\begin{equation}
    |x^*\rangle=B^*_0(x)\sum_{\ell =0}^\infty \frac{(x-1)^{\ell}}{a_0\dots a_{\ell -1}}\; |\beta, \ell \rangle.\label{eigvecK2T}
\end{equation}
Here also we do not fix the initial value $B^*_0(x)$ and the norm of $|x^*\rangle$. 
\end{enumerate}

\subsubsection{The Jacobi polynomials $J_n^{(a,b)}(x)$ as overlaps between eigenvectors}

We shall understand $\langle x |$ as $| x^*\rangle ^{\top}$. From the expressions \eqref{eigvecK1} and \eqref{eigvecK2T} of the eigenvectors of $K_1$ and $K_2^{\top}$, the overlaps $\langle x|\mu_n^{(\alpha,\beta)}\rangle$ are readily found to be:

\begin{equation}
    \langle x|\mu_n^{(\alpha,\beta)}\rangle = A_{n,0} B_0^*(x)\sum_{k,\ell = 0}^\infty \frac{a_0 \dots a_{k-1}}{a_0\dots a_{\ell -1}} (x-1)^{\ell} (-1)^k \frac{(-n)_k}{k!} \frac{(n+\alpha +\beta +1)_k}{(\beta +1)_k} \langle \beta, \ell|\beta, k\rangle.\label{over}
\end{equation}

\noindent Recalling that $\langle \beta, \ell|\beta, k\rangle = \delta _{\ell,k}$, equation \eqref{over} simplifies to:
\begin{align}
     \langle x|\mu_n^{(\alpha,\beta)}\rangle &= A_{n,0} B_0^*(x)\sum_{k=0}^\infty \frac{(-n)_k (n+\alpha+\beta+1)_k}{(\beta + 1)_k} \frac{(1-x)^k}{k!} \nonumber\\
     &= A_{n,0} B_0^*(x) \hg{2}{1}\argu{-n,n+\alpha +\beta+1}{\beta+1}{1-x}.
\end{align}
 With the help of the following Kummer relation \cite{bateman1953higherI}:
 \begin{equation}
     \hg{2}{1}\argu{-n, b}{c}{x} = \frac{(c-b)_n}{(c)_n}\hg{2}{1}\argu{-n,b}{-n+b+1-c}{1-x}, \label{Kummer}
\end{equation}
one arrives at
\begin{equation}
    \langle x|\mu_n^{(\alpha,\beta)}\rangle = (-1)^n \frac{n!}{(\beta +1)_n} A_{n,0} B_0^*(x) J_n^{(\alpha,\beta)}(x).\label{overJac}
\end{equation}
Hence, up to normalizations, the overlaps between the eigenvectors of the generators $K_1$ and $K_2$ of the Jacobi algebra of rank one \eqref{rank1} with parameters $\alpha$ and $\beta$, are given by the Jacobi polynomials $J_n^{(\alpha,\beta)}(x)$ with $-n(n+\alpha + \beta + 1)$ the eigenvalues of $K_1$ and the variable $x$ corresponding to those of $K_2$.

\subsubsection{Differential and difference realizations of $\mathfrak{J}_1$}

A differential and difference realization of the Jacobi algebra $\mathfrak{J}_1$ which encodes the bispectral properties of these polynomials are provided by the differential equation and recurrence relation that these polynomials obey. They read:
\begin{enumerate}[leftmargin=*]
    \item Differential equation
\begin{align}\label{unide}
  &H_x^{(a,b)}[J_{n}^{(a,b)}(x)]=-n\,(n+a+b+1)J_{n}^{(a,b)}(x),
\end{align}
where
\begin{align}
H_x^{(a,b)} = x(1-x)\partial_x^2 + (a+1 - (a+b+2)x)\partial_x.
\end{align}

\item Recurrence relation
\begin{align}
(1-2x) &J_n^{(a,b)}(x)
= \Hat{H}^{(a,b)} J_n^{(a,b)}(x),
\end{align}
where
\begin{align}
\Hat{H}^{(a,b)}&=
\frac{2(n+1)(n+a+b+1)}{(2n+a+b+1)(2n+a+b+2)}S_+\nonumber  \\
&-\frac{(a-b)(a+b)}{(2n+a+b)(2n+a+b+2)}I+\frac{2(n+a)(n+b)}{(2n+a+b)(2n+a+b+1)} S_-,
\label{RR:univariate}
\end{align}
and the shift operators $S_{\pm}$ on the variable $n$ act according to $S_{\pm}f_n=f_{n\pm 1}$.
\end{enumerate}
\medskip
Take
\begin{equation}
   \ K_1=\rho_n^{(a,b)}(x)H_x^{(a,b)}[\rho_n^{(a,b)}(x)]^{-1},\qquad K_2=x,
\end{equation}
with
\begin{equation}
    \rho_n^{(a,b)}(x)=\left[\frac{x^{a}(1-x)^{b}}{N_n^{(a,b)}}\right]^\frac{1}{2}.
\end{equation}
It is directly checked that this identification provides a differential realization of the commutation relations \eqref{rank1} that define the Jacobi algebra of rank one with $\alpha =a$ and $\beta =b$. In this model, the normalized eigenbases are given by the following functions:
\begin{equation}
    |\mu_n^{(\alpha,\beta)}\rangle \; \longrightarrow \; \rho_n^{(\alpha,\beta)}(x)J_n^{(\alpha,\beta)}(x), \qquad
    |x\rangle \; \longrightarrow \; \delta_x,
\end{equation}
with $\delta_x$ Dirac's delta function: $\delta _x(y) = \delta (x-y)$.

\begin{remark}
    This result fixes the range of the eigenvalue $x$ to be $[0,1]$ and $\alpha,\beta>-1$ to ensure the orthogonality:
    \begin{equation}
 \langle  \mu_m^{(\alpha,\beta)} |\mu_n^{(\alpha,\beta)}\rangle=\delta_{n,m} \; \longrightarrow \; \int_0^1 \rho_n^{(\alpha,\beta)}(x)\rho_m^{(\alpha,\beta)}(x)J_n^{(\alpha,\beta)}(x)J_m^{(\alpha,\beta)}(x) dx=\delta_{n,m},
\end{equation}
where we have applied equation \eqref{orth1var}.
\end{remark}


\subsection{The Racah algebra}

\subsubsection{Definition}

The rank one Racah algebra $\mathfrak{R}_1$ has two generators, again denoted by $K_1$ and $K_2$, that are subjected to the relations \cite{genest2014superintegrability}, \cite{genest2014racah}:
\begin{subequations}\label{Ract}
  \begin{align}
    [[K_1,K_2],K_1]=\; &2K_1^2 + 2\{K_1,K_2\}+ \xi K_1+ \eta_1 K_2+ \zeta_1, \label{Rac1}\\
    [K_2,[K_1,K_2]]=\; &2K_2^2 + 2\{K_1,K_2\} + \xi K_2+ \eta_2 K_1+ \zeta_2,\label{Rac2}
\end{align}  
\end{subequations}
where $\xi, \eta_1, \eta_2, \zeta_1, \zeta_2$ are parameters.
(Note that the parameters $\eta_1$ and $\eta_2$ can be eliminated by adding appropriate constants to the generators; this is often done to standardize the presentation but it will be more convenient to not do so here.)
Its representations and connections to the orthogonal polynomials from which it takes its name can be carried out in a fashion similar to the treatment given above of the rank one Jacobi algebra. We shall not review this here as it would be lengthy and can be found in the literature \cite{genest2014superintegrability}, \cite{genest2014racah}, \cite{huang2020finite}. As just done in the Jacobi case, we shall however indicate how $\mathfrak{R}_1$ is realized by the bispectral operators of the Racah polynomials. This will show that the overlaps between certain representation bases are given by these polynomials.

\subsubsection{The Racah polynomials and their orthogonality relations}

The Racah polynomials denoted by $R_n(\lambda(\ell); \alpha, \beta, \gamma, \delta) $ are explicitly defined as follows \cite{koekoek2010hypergeometric}:
\begin{equation}
R_n(\lambda^{(\gamma,\delta)}(\ell); \alpha, \beta, \gamma, \delta) 
= {}_4F_3\left( \begin{array}{c}
-n,\ n+\alpha+\beta+1,\ -\ell,\ \ell+\gamma+\delta+1 \\
\alpha+1,\ \beta+\delta+1,\ \gamma+1
\end{array} ; 1 \right), \quad n = 0, 1, 2, \ldots, N, \label{defRac}
\end{equation}
where
\begin{equation}
  \lambda^{(\gamma,\delta)}(\ell) = \ell(\ell + \gamma + \delta + 1),  
\end{equation}
\begin{equation}
   \alpha + 1 = -N \quad \text{or} \quad \beta + \delta + 1 = -N \quad \text{or} \quad \gamma + 1 = -N, 
\end{equation}
with \( N \) a nonnegative integer.
In this and the following subsections, $R_n(\lambda(\ell))$ (resp. $\lambda(\ell)$) stands for $R_n(\lambda(\ell); \alpha, \beta, \gamma, \delta)$ (resp. $\lambda^{(\gamma,\delta)}(\ell)$).

They satisfy the following orthogonality relations
\begin{equation}
\sum_{\ell=0}^N
w^{(\alpha, \beta, \gamma, \delta)}(\ell)\;
R_m(\lambda(\ell)) \, R_n(\lambda(\ell))  = M_n^{(\alpha, \beta, \gamma, \delta)}
\delta_{mn},
\end{equation}
where\footnote{Note that the identity $\frac{(\gamma + \delta +1)_\ell \;\left(\frac{\gamma + \delta +3}{2}\right)_\ell}{\left(\frac{\gamma + \delta +1}{2}\right)_\ell} = \frac{(\gamma + \delta +2)_{2\ell}}{(\gamma+\delta +\ell+1)_\ell}$ was used to transform the formula for the weight $ w^{(\alpha, \beta, \gamma, \delta)}(\ell)$ given in \cite{koekoek2010hypergeometric} into \eqref{weight}.}
\begin{equation}
    w^{(\alpha, \beta, \gamma, \delta)}(\ell)=\frac{
(\alpha +1)_\ell\, (\beta +\delta +1)_\ell\, (\gamma +1)_\ell\, (\gamma +\delta +2)_{2\ell}
}{
(-\alpha +\gamma +\delta +1)_\ell\, (-\beta +\gamma +1)_\ell\, (\gamma +\delta +\ell +1)_\ell\, (\delta +1)_\ell\, \ell! \label{weight}
},
\end{equation}

\begin{equation}
    M_n^{(\alpha, \beta, \gamma, \delta)} = K\frac{
 (n+\alpha +\beta +1)_n\, (\alpha +\beta -\gamma +1)_n\, (\alpha -\delta +1)_n\, (\beta +1)_n\, n!
}{
(\alpha +\beta +2)_{2n}\, (\alpha +1)_n\, (\beta +\delta +1)_n\, (\gamma +1)_n
},\label{normal}
\end{equation} 
and
\begin{equation}
K = \begin{cases}
\displaystyle
\frac{(-\beta)_N\, (\gamma + \delta + 2)_N}{(-\beta + \gamma + 1)_N\, (\delta + 1)_N}
& \text{if } \alpha + 1 = -N, \\[1em]
\displaystyle
\frac{(-\alpha + \delta)_N\, (\gamma + \delta + 2)_N}{(-\alpha + \gamma + \delta + 1)_N\, (\delta + 1)_N}
& \text{if } \beta + \delta + 1 = -N, \\[1em]
\displaystyle
\frac{(\alpha + \beta + 2)_N\, (-\delta)_N}{(\alpha - \delta + 1)_N\, (\beta + 1)_N}
& \text{if } \gamma + 1 = -N.
\end{cases}
\label{K}
\end{equation}

\noindent Observe that 
\begin{equation}
     w^{(\gamma, \delta, \alpha, \beta)}(n) = K \left[  M_n^{(\alpha, \beta, \gamma, \delta)}\right]^{-1}.
\end{equation}
Clearly the orthonormalized functions
\begin{align}
    S_n^{(\alpha, \beta, \gamma, \delta)}(\ell)=&\sqrt{\frac{ w^{(\alpha, \beta, \gamma, \delta)}(\ell)}{  M_n^{(\alpha, \beta, \gamma, \delta)}}}R_n(\lambda(\ell)) \nonumber \\
    & =\sqrt{\frac{w^{(\alpha, \beta, \gamma, \delta)}(\ell)\;w^{(\gamma, \delta, \alpha, \beta)}(n)}{K}} \;R_n(\lambda(\ell)),\label{orthfn}
\end{align}
satisfy
\begin{equation}
    \sum_{\ell=0}^N \;  S_m^{(\alpha, \beta, \gamma, \delta)}(\ell)\; S_n^{(\alpha, \beta, \gamma, \delta)}(\ell) \;= \;\delta_{m,n}. \label{orthS}
\end{equation}
Note that the Racah polynomials $R_n(\lambda(\ell))$ as well as the functions $S_n^{(\alpha, \beta, \gamma, \delta)}(\ell)$ are invariant under the exchanges 
\begin{equation}
    n \leftrightarrow \ell, \qquad \alpha \leftrightarrow \gamma, \qquad \beta \leftrightarrow \delta,
\end{equation}
and are in this sense self-dual under the exchanges of the variable and the degree.

\subsubsection{Bispectrality properties of the Racah polynomials}

We here record the two eigenvalue equations that the Racah polynomials verify \cite{koekoek2010hypergeometric}. 
\begin{enumerate}[leftmargin=*]    \item Three-term recurrence relation
    \begin{equation}
\lambda(\ell) R_n(\lambda(\ell)) = A_n R_{n+1}(\lambda(\ell)) - (A_n + C_n) R_n(\lambda(\ell)) + C_n R_{n-1}(\lambda(\ell)),
\end{equation}
where $A_n$ and $C_n$ are given by
\begin{align}
&   A_n = \frac{(n + \alpha + 1)(n + \alpha + \beta + 1)(n + \gamma + 1)(n + \beta + \delta + 1)}
{(2n + \alpha + \beta + 1)(2n + \alpha + \beta + 2)},\nonumber\\ 
&C_n = \frac{n(n + \beta)(n + \alpha - \delta)(n + \alpha + \beta - \gamma)}
{(2n + \alpha + \beta)(2n + \alpha + \beta + 2)}.
\end{align}

\item Difference equation
\begin{equation}
\left[B(\ell) \mathcal{T}^+ - (B(\ell) + D(\ell)) + D(\ell) \mathcal{T}^-\right] R_n(\lambda(\ell)) = n(n + \alpha + \beta + 1) R_n(\lambda(\ell)),
\end{equation}
with $\mathcal{T}^{\pm}f(\ell) = f(\ell\pm 1)$ and where $B(\ell)$ and $D(\ell)$ are
\begin{align}
B(\ell) &= \frac{(\ell + \alpha + 1)(\ell + \beta + \delta + 1)(\ell + \gamma + 1)(\ell + \gamma + \delta + 1)}{(2\ell + \gamma + \delta + 1)(2\ell + \gamma + \delta + 2)}, \nonumber \\
D(\ell) &= \frac{\ell(\ell - \alpha + \gamma + \delta)(\ell - \beta + \gamma)(\ell + \delta)}{(2\ell + \gamma + \delta)(2\ell + \gamma + \delta + 1)}. 
\end{align}

\end{enumerate}

\subsubsection{A realization of $\mathfrak{R}_1$ in terms of difference operators \label{subs:relR}}

The bispectral operators provide a representation of the Racah algebra. Take
\begin{align}
K_1 &= -B(\ell) T^+ - D(\ell) T^- + (B(\ell) + D(\ell)), \nonumber \\
K_2 &= -\ell(\ell + \gamma + \delta + 1). \nonumber 
\end{align}
A direct computation shows that these operators $K_1$, $K_2$ verify the defining relations \eqref{Ract} of $\mathfrak{R}_1$ with structure parameters
\begin{align}
 \xi \;&= \beta(\delta - \gamma - 2) - \alpha(2\beta + \gamma + \delta + 2) - 2(\gamma + 1)(\delta + 1), \label{xiRac}\\
\eta_1 &= -(\alpha + \beta)(2 + \alpha + \beta),\label{eta1Rac}\\
\eta_2 &= -(\gamma + \delta)(2 + \gamma + \delta), \label{eta2Rac}\\
\zeta_1 &= (\alpha + 1)(\alpha + \beta)(\beta + \delta + 1)(\gamma + 1), \label{zeta1Rac}\\
\zeta_2 &= (\alpha + 1)(\beta + \delta + 1)(\gamma + 1)(\gamma + \delta)\label{zeta2Rac}.
\end{align}

This realization of $\mathfrak{R}_1$ on functions of $\ell$ provides a faithful model for the properties of representation bases of this algebra. On the one hand, there is a basis $\{|n;\alpha, \beta, \gamma, \delta \rangle,\; \, n=0\,\dots,N\}$ given up to normalizations by the Racah polynomials $R_n(\lambda(k); \alpha, \beta, \gamma, \delta) $. The generator $K_1$ is diagonal in this basis while $K_2$ is tridiagonal. On the other hand, we have a second basis $\{|e_{\ell}\rangle,\; \ell = 0,\dots, N \}$ realized by $e_{\ell}=\delta_{\ell,k}$ in which $K_2$ is diagonal and $K_1$ tridiagonal. 

If the basis vectors are normalized, we thus infer that the overlaps of the eigenbases defined by these properties and hence forming a Leonard pair \cite{terwilliger2001two}, will read in terms of the orthonormalized functions \eqref{orthfn}:
\begin{align}
    \langle e_\ell \;|\;n; \alpha, \beta, \gamma, \delta \rangle &= \;S_n^{(\alpha, \beta, \gamma, \delta)}(\ell) = \;S_\ell^{(\gamma, \delta, \alpha, \beta)}(n)\nonumber \\
    &=\sqrt{\frac{w^{(\alpha, \beta, \gamma, \delta)}(\ell)}{M_n^{(\alpha, \beta, \gamma, \delta)}}}R_n(\lambda^{(\gamma,\delta)}(\ell); \alpha, \beta, \gamma, \delta) ,\label{ove_R}
\end{align} 
where $w^{(\alpha, \beta, \gamma, \delta)}(\ell)$ is given by \eqref{weight} and $M_n^{(\alpha, \beta, \gamma, \delta)}$ by \eqref{normal} and \eqref{K}. This fact is corroborated by the general representation theory \cite{genest2014superintegrability}, \cite{genest2014racah}. Summing up, the overlaps between the eigenvectors of $K_1$ with eigenvalues $-n(n+\alpha + \beta + 1)$ and those of $K_2$ with eigenvalues $-\ell (\ell + \gamma +\delta + 1)$ are given by \eqref{ove_R}. The selfduality of the Racah polynomials is reflected in the symmetry of the algebra relations  under the exchanges $K_1 \leftrightarrow K_2$ and of the indices $1$ and $2$ in the parameters $\eta_i \;\text{and}\: \;\zeta_i,\; i=1,2$.

\section{Structural elements of the rank two Jacobi algebra \label{sec:2Jacobi}}

The defining relations of the rank two Jacobi algebra $\mathfrak{J}_2$ upon which the interpretation of the two-variable Jacobi polynomials will be based are given in Appendix \ref{APP A}. We here focus on certain structural features of $\mathfrak{J}_2$. This algebra is generated by five elements namely, $L, L_1, L_3, X_1, X_3$. The centralizers of each of these generators form significant subalgebras of $\mathfrak{J}_2$ that will be at the heart of our study. 

From \eqref{zero}, it is seen that the centralizer of every element of the generating set of $\mathfrak{J}_2$ is two-generated. For each element, their generators are:
\begin{subequations}\label{sat}
\begin{align}
    L_1 \qquad :& \qquad \{L, X_1\}  \label{sa1} \\
    X_1 \qquad :& \qquad \{L_1, X_3\} \\
  X_3 \qquad :& \qquad \{X_1, L_3\} \\
   L_3 \qquad :& \qquad \{X_3,L\}\\
    L \qquad :& \qquad \{L_3,L_1\}.\label{sa5}
\end{align}    
\end{subequations}
\noindent We shall describe in turn each of these subalgebras.

\subsection{The centralizer $C_{L_1}(\mathfrak{J}_2)$ of $L_1$ \label{sub3.1}} 
As indicated in the list above, the generators that commute with $L_1$ are $L$ and $X_1$.
From \eqref{LX1L} and \eqref{LX1X1} in Appendix \ref{APP A}, we see that the commutation relations involving $L$ and $X_1$ are:
\begin{subequations}\label{LX1Ltt}
 \begin{align}
     [[L,X_1]\;,L\;]&=2\{X_1,L\}-2L+2L_1-(a+b+c+1)((a+b+c+3)X_1 -(a+1)I),\label{LX1Lt}\\
     [[L,X_1],X_1]&=-2X_1^2+2X_1.\label{LX1X1t}  
 \end{align}    
\end{subequations}

Knowing that $L_1$ is central since $[L,L_1]=0$ and $[X_1,L_1]=0$, this defines an algebra that we wish to identify. To that end write the central $L_1$ as
\begin{equation}
       L_1=-k (k+ b+c+1)I.
\end{equation}
Then, upon taking 
\begin{equation}
    K_1=L+k(k+a+b+c+2)I \qquad \text{and } \qquad K_2=X_1,
\end{equation}
we recognize that the commutation relations \eqref{LX1Ltt} are transformed into those given in \eqref{rank1} of $\mathfrak{J}_1$ with 
\begin{equation}
    \alpha=a \qquad \text{and} \qquad \beta= 2k+b+c+1.
\end{equation}
This establishes that $C_{L_1}(\mathfrak{J}_2)$ is a rank one Jacobi algebra centrally extended by the presence of $L_1$ via $k$ in the structure parameters.

\subsection{The centralizer $C_{X_1}(\mathfrak{J}_2)$ of $X_1$ \label{subs:X1}}
The generators that commute with $X_1$ are $L_1$ and $X_3$. The commutation relations \eqref{L1X3L1} and \eqref{L1X3X3} from Appendix \ref{APP A} that involve these operators are:
\begin{subequations}\label{L1X3L1tt}
\begin{align}   
 [[L_1,X_3],L_1]&=2\{X_3 ,L_1\}+\{X_1,L_1\}-2L_1  
 -(b+c) \left((b+c+2)X_3 -(c+1)(I-X_1) \right),\label{L1X3L1t}\\
 [[L_1,X_3],X_3]&=-2X_3^2+2(I-X_1)X_3.\label{L1X3X3t}
 \end{align}    
\end{subequations}

\noindent With $X_1$ central these relations define an algebra. Assume $X_1-I$ to be invertible and set
\begin{equation}
    K_1=L_1 \qquad \text{and} \qquad K_2 = \frac{X_3}{X_1 -I}\;+I \qquad \text{with} \qquad \alpha = b, \qquad \beta =c.\label{X1central}
\end{equation}
It is then immediate to see that relations \eqref{rank1} are transformed into 
\eqref{L1X3L1tt}. Hence $C_{X_1}(\mathfrak{J}_2)$ is found to be a rank one Jacobi algebra with $X_1$ as a central element under the identification of the generators $K_1$ and $K_2$ of $\mathfrak{J}_1$ given in \eqref{X1central}.

\medskip
The situation will be parallel with the centralizers of $L_3$ and of $X_3$.

\subsection{The centralizer $C_{L_3}(\mathfrak{J}_2)$ of $L_3$\label{subs:L3}}

The generators that commute with $L_3$ are $L$ and $X_3$. These operators are seen to form a centrally extended Jacobi algebra of rank one in a fashion similar to $L$ and $X_1$. The commutation relations \eqref{LX3L} and \eqref{LX3X3} from Appendix \ref{APP A} that involves only that pair are:
\begin{subequations}\label{LX3Ltt}
   \begin{align}
    [[L,X_3]\;,L\;]&=2\{ X_3,L\}-2L+2 L_3
    -(a+b+c+1)((a+b+c+3)X_3 - (c+1)I),\label{LX3Lt}\\
    [[L,X_3],X_3]&=-2X_3^2+2X_3,\label{LX3X3t}
\end{align} 
\end{subequations}
and with $L_3$ central they define an algebra. To identify it, write the central $L_3$ in the form
\begin{equation}
    L_3=-k (k+a+b+1)I.
\end{equation}
It is quite straightforward to see that the relations \eqref{LX3Ltt} between $L$ and $X_3$ are transformed into the relations \eqref{rank1} of $\mathfrak{J}_1$ with
\begin{equation}
    K_1=L+k (k+a+b+c+2)I \qquad \text{and} \qquad K_2=X_3, \label{idenCL3}
\end{equation}
and the parameters given by
\begin{equation}
    \alpha =c \qquad \text{and} \qquad \beta=a+b+2k+1. \label{parCL3}
\end{equation}
We thus conclude that $C_{L_3}(\mathfrak{J}_2)$ is a centrally extended rank one Jacobi algebra.

\subsection{The centralizer $C_{X_3}(\mathfrak{J}_2)$ of $X_3$ \label{subs:X3}}

The generators of this subalgebra, i.e., the operators that commute with $X_3$, are $L_3$ and $X_1$. Their mutual commutation relations are given by equations \eqref{L3X1L3} and \eqref{L3X1X1} of Appendix \ref{APP A} and read:
\begin{subequations}\label{L3X1L3tt}
    \begin{align}
    [[L_3,X_1],L_3]&=\{2X_1+X_3-I,L_3\} 
 -(a+b)\left((a+1)(X_1+X_3-I)+(b+1)X_1\right), \label{L3X1L3t}\\
    [[L_3,X_1],X_1]&=-2X_1(X_1+X_3-I). \label{L3X1X1t}
\end{align}
\end{subequations}
 \noindent Assume $X_3-I$ to be invertible, positing 
\begin{equation}
    K_1=L_3 \qquad\text{and}\qquad K_2=\frac{X_1}{X_3-I} +I, \label{JacsubL3X1}
\end{equation}
allows to rewrite \eqref{L3X1L3tt} precisely as the commutation relations \eqref{rank1} of $\mathfrak{J}_1$ with the parameters given by
\begin{equation}
    \alpha=b, \; \beta=a. \label{parCX3}
\end{equation}
As was the case in \eqref{X1central}, the division by a generator is allowed as it is central. The conclusion is that $C_{X_3}(\mathfrak{J}_2)$ is also a Jacobi algebra of rank one.

\medskip

All the centralizers examined up to this point have turned out to be isomorphic to (central extensions) of $\mathfrak{J}_1$. This will not be so for the fifth one.

\subsection{The centralizer $C_{L}(\mathfrak{J}_2)$ of $L$ \label{subs:L}}

Generated by the elements $L_1$ and $L_3$ that commute with $L$,  $C_{L}(\mathfrak{J}_2)$ is by definition the symmetry algebra of $L$. The commutation relations from Appendix \ref{APP A} that involve only $L_1$ and $L_3$ (apart from the central $L$) are \eqref{L1L3L1} and \eqref{L1L3L3}:
\begin{subequations}\label{L1L3L1tt}
\begin{align}
[[L_1,L_3],L_1]&=2\{L_1,L_3\}+2L_1^2-2L_1L +(b+c)(b+1)(L-L_1-L_3) \nonumber\\&
    - (b+c)(c+1)L_3-(b-c)(a+1)L_1, \label{L1L3L1t}\\
[[L_1,L_3],L_3]&=-2\{L_1,L_3\}-2L_3^2+2L_3L -(a+b)(b+1)(L-L_1-L_3) \nonumber\\&
    + (a+b)(a+1)L_1+(b-a)(c+1)L_3. \label{L1L3L3t}
 \end{align}    
\end{subequations}
\noindent Comparing the equations \eqref{L1L3L1tt} with the relations \eqref{Ract} of $\mathfrak{R}_1$, we observe under the identification $K_1 \rightarrow L_1$ and $K_2 \rightarrow L_3$ that the operators $L_1$ and $L_3$ realize together a central extension of the rank one Racah algebra, due to the presence of $L$ in the parameters, that are given by:
\begin{align}
    \xi=& \;-(b+c)(b+1)-(b-c)(a+1)-2L=-(a+b)(b+1)-(b-a)(c+1)-2L, \label{xiJac}\\
    \eta _1=& \;-(b+c)(b+c+2),\label{eta1Jac}\\
    \eta _2=& \;-(a+b)(a+b+2), \label{eta2Jac}\\
    \zeta _1=&\;(b+c)(b+1)L, \label{zeta1Jac}\\
    \zeta _2=&\;(a+b)(b+1)L \label{zeta2Jac}.  
\end{align}

We thus see that the centralizer $C_{L}(\mathfrak{J}_2)$ which is the symmetry algebra of $L$, is a central extension of the rank one Racah algebra.

\subsection{A pentagonal depiction}
Referring to the list \eqref{sat}, the first four centralizers are isomorphic to the (centrally extended) Jacobi algebra $\mathfrak{J}_1$ of rank one, while the last one is a central extension of the Racah algebra $\mathfrak{R}_1$ of rank one. It proves telling to organize with a pentagon the subalgebra structure of the rank two Jacobi algebra $\mathfrak{J}_2$ that we have found. 

A feature of the ordered list \eqref{sa1}-\eqref{sa5} is that the second generator of a given centralizer is central in the subsequent one and this property is cyclic. Start the sequence with the second element, $L_1$, from the centralizer of $L$ and associate in a clockwise fashion to each vertex of a pentagon, the pairs formed by the last and first generators of consecutive centralizers. This gives the graph of Figure \ref{fig:tr}. At each corner of the pentagon, there is a pair of commuting generators. The picture that emerges is that the elements at the boundaries of straight line edges generate a rank one Jacobi algebra while those at the boundaries of the dotted edge at the bottom generate a rank one Racah algebra. The commuting pairs attached to each node can be viewed as defining bases of joint eigenvectors for representation spaces of the Jacobi algebra of rank two. The transformation matrices between adjacent bases along the pentagon will be expressed in terms of the polynomials associated to the algebra assigned to the edge connecting the commuting elements/eigenbases. A synthetic underpinning of the algebraic interpretation that we shall explain next of the two-variable Jacobi polynomials through representations of $\mathfrak{J}_2$ is thus seen to lie in this pentagonal diagram.

\begin{figure}[htbp] \label{fig:1}
\begin{center}
\begin{minipage}[t]{.9\linewidth}
\begin{center}
\begin{tikzpicture}[scale=1.3]
   
    \def\R{2cm}
    \def\Rt{2.1cm}
    \def\Rtt{1.85cm}

    \foreach \i in {1,...,5} {
        \coordinate (P\i) at ({90 + (\i-1)*360/5}:\R);
    }
    \draw (P3)--(P2) -- (P1) -- (P5) -- (P4);
    \draw[dashed] (P3) -- (P4);


\node at ({90 + (+0.12)*360/5}:\Rt) {$X_3$};
\node at ({90 + (-0.12)*360/5}:\Rt) {$X_1$};
\node at ({90 + (-0.12+1)*360/5}:\Rt) {$L_1$};
\node at ({90 + (+0.12+1)*360/5}:\Rt) {$X_1$};
\node at ({90 + (-0.12-1)*360/5}:\Rt) {$X_3$};
\node at ({90 + (+0.12-1)*360/5}:\Rt) {$L_3$};
\node at ({90 + (-0.12-2)*360/5}:\Rt) {$L_3$};
\node at ({90 + (+0.12-2)*360/5}:\Rt) {$L$};
\node at ({90 + (-0.12+2)*360/5}:\Rt) {$L$};
\node at ({90 + (+0.12+2)*360/5}:\Rt) {$L_1$};

     \foreach \i in {1,...,5} {
        \draw[fill]  (P\i) circle (0.08); ;
     }

\end{tikzpicture}
\caption{Two generators around the same vertex commute; two elements at the boundaries of the same solid line generate a Jacobi algebra of rank one and $L_1$, $L_3$ at the boundaries of the dashed edge generate a Racah algebra of rank one.   \label{fig:tr}}

\end{center}
\end{minipage}
\end{center}
\end{figure}
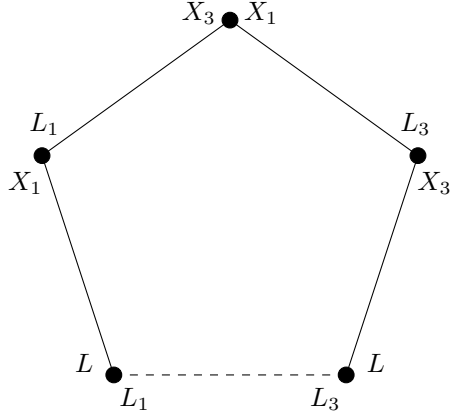

\section{An algebraic interpretation of the two-variable Jacobi polynomials on the triangle \label{sec:alge}}

We shall now proceed to explain how the two-variable Jacobi polynomials arise from the representation of $\mathfrak{J}_2$ given in Appendix \ref{sec:rep} where we observed that all generators are symmetrizable and their sets of eigenvectors therefore orthonormalizable. It is customary to define representation bases as joint eigenvectors of the elements of maximal abelian subalgebras. The centralizers discussed in the last section that are all made out of two commuting generators will hence play key roles from that standpoint in the examination of representations of the rank two Jacobi algebra and in fact, all the corresponding bases will be used. From looking at the pentagon picture, bases of particular interest are those corresponding to the top and bottom of the drawing.  
\begin{itemize}
\item As shown previously, the generators $X_1$ and $X_3$ are always associated to the generator $K_2$ of different rank one Jacobi subalgebras. Therefore, the eigenvectors associated to the top vertex are the vectors $|x,y\rangle$ which we take to satisfy the following eigenvalue equations:
\begin{equation}\label{u}
    X_1 \;|x,y\rangle = x\;|x,y\rangle, \qquad \text{and} \qquad X_3\; |x,y\rangle = (1-x-y) \; |x,y\rangle.
\end{equation}
The knowledge of the spectra of the corresponding operators in representations of the rank one Jacobi algebra specifies the possible values of $x$ and $y$.
The range of $x$ is $0\le x\le 1$ and, due to the identification \eqref{X1central},  $\frac{1-x-y}{x-1}+1$ is between 0 and 1 which leads to $ x \le 1-y \le1$.
\item There are two sets of eigenvectors corresponding to the bottom corners of the pentagon. The one associated to the left vertex will be formed by the vectors $|n,k;a,b,c\rangle$ which obey the two eigenvalue equations:
\begin{subequations}\label{blt}
    \begin{align} 
    &L \; |n,k;a,b,c\rangle = -n(n+a+b+c+2) \;|n,k;a,b,c\rangle, \label{bl} \\
    &L_1 \; |n,k;a,b,c\rangle = -k(k+b+c+1) \; |n,k;a,b,c\rangle. \label{bl1}
\end{align}
\end{subequations}
The choice for the eigenvalues comes from the fact that we know that $L$ and $L_1$ are part of Jacobi algebras of rank one whose generators $K_1$ have spectra of that type. From the one described in Subsection \ref{subs:X1}, one deduces that $k$ take values in $\mathbb{N}$ whereas the one in Subsection \ref{sub3.1} leads to eigenvalues of the following form, for $m\in \mathbb{N}$
\begin{align}
    -m(m+a+b+c+2+2k)-k(k+a+b+c+2)=-(m+k)(m+k+a+b+c+2)\,.
\end{align}
Therefore, one gets that $0\le k\le n=m+k$.
\medskip

\noindent The eigenvectors corresponding to the right corner will have for elements the vectors $|n,k;a,b,c\rangle ^\pi$ ($0\le k \le n$) that are solutions of the two eigenvalue equations:
\begin{subequations}\label{brt}
    \begin{align}
    &L \; |n,k;a,b,c\rangle ^\pi = -n(n+a+b+c+2) \;|n,k;a,b,c\rangle ^\pi, \label{br}\\
    &L_3 \; |n,k;a,b,c\rangle ^\pi = -k(k+a+b+1) \; |n,k;a,b,c\rangle ^\pi.\label{br1}
\end{align}
\end{subequations}
\end{itemize}

\noindent All these eigenvectors will be orthonormalized. We shall now endeavor to show that the overlaps  between the eigenvectors defined by \eqref{u} and the ones specified by \eqref{blt}
are proportional to the standard two-variable Jacobi polynomials $J_{n,k}^{(a,b,c)}(x,y)$. This will be our main result. Similarly, we shall see that the overlaps between the eigenvectors associated to the upper vertex of the pentagon \eqref{u} and the ones attached to the right lower corner \eqref{brt} 
are proportional to the polynomials $J_{n,k}^{(c,b,a)}(1-x-y,y)$.

\subsection{The two-variable polynomials $J_{n,k}^{(a,b,c)}(x,y)$ as overlaps between representation bases of $\mathfrak{J}_2$ \label{subs:1stJ}}

The algebraic interpretation of the two-variable Jacobi polynomials $J_{n,k}^{(a,b,c)}(x,y)$ will be obtained by computing the overlaps  $\langle x, y|n,k;a,b,c\rangle$ between the eigenvectors defined in \eqref{u} and \eqref{blt}. Consider the additional representation basis $\{|x,k;b,c\rangle\}$ which is formed by the joint eigenvectors of $X_1$ and $L_1$:
\begin{equation} 
    X_1\;|x,k;b,c\rangle=x\;|x,k;b,c\rangle, \qquad L_1\;|x,k;b,c\rangle=-k(k+b+c)\;|x,k;b,c\rangle. \label{intbas}
\end{equation}
To obtain the expansion coefficients of the basis vectors $|x,y\rangle$ in terms of the basis vectors $|n,k;a,b,c\rangle$ (or vice versa), we can proceed in a telescoped way by convoluting the overlaps between the bases $\{|x,y\rangle\}$ and $\{|x,k;b,c\rangle\}$ with the overlaps between the bases $\{|x,k;b,c\rangle\}$ and $\{|n,k;a,b,c\rangle\}$. That is, assuming that the basis $\{|x,k;b,c\rangle\}$ is orthonormalized, we shall perform the computation by using:
\begin{equation}
    \langle n,k;a,b,c\;|\;x,y\rangle = \sum_{k'} \int _0^1 dx' \langle n,k;a,b,c\;|\;x',k';b,c\rangle\langle x',k';b,c\;|\;x,y\rangle. \label{conv1}
\end{equation}
This corresponds to following the path on the left side of the pentagon to reach the bottom from the top. This path is depicted in Figure \ref{fig:2} and will underlie the representation-theoretic interpretation of the polynomials $J_{n,k}^{(a,b,c)}(x,y)$.
\begin{figure}[ht] 
\centering
\begin{tikzpicture}[scale=1.5]
  \coordinate (A) at (0,0);
  \coordinate (B) at (2,0);
  \coordinate (C) at (4,0);

  \fill (A) circle (2pt);
  \fill (B) circle (2pt);
  \fill (C) circle (2pt);

  \draw (A) -- (B) -- (C);

  \node at ($(A)+(0,0.4)$) {\(X_1,\,X_3\)};
  \node at ($(B)+(0,0.4)$) {\(L_1,\,X_1\)};
  \node at ($(C)+(0,0.4)$) {\(L,\,L_1\)};
\end{tikzpicture}
\caption{The left hand side of the pentagon of Figure \ref{fig:tr} that connects the $\{|x,y\rangle\}$ basis to the $\{|n,k;a,b,c\rangle\}$ basis via the $\{|x,k;b,c\rangle\}$ basis\label{fig:2}}.
\end{figure}

\begin{pr} \label{prop 1}

 Let $\{|x,y\rangle\ |\  0\le x\le 1,0\le y \le 1-x\}$ and $\{|n,k;a,b,c\rangle\ |\ n,k\in \mathbb{N}, 0\le k \le n\}$ be the two representation bases of $\mathfrak{J}_2$ defined by equations \eqref{u} and \eqref{blt}. The two-variable Jacobi polynomials $J_{n,k}^{(a,b,c)}(x,y)$ defined in \eqref{defJac2} arise in this algebraic context as overlaps between those bases; namely,
 \begin{equation}
     \langle n,k;a,b,c\;|\;x,y\rangle =\sqrt{\frac{x^a y^b (1-x-y)^c}{N_k^{(b,c)}N_{n-k}^{(a,b+c+2k+1)}}}  J_{n,k}^{(a,b,c)}(x,y).
 \end{equation}
\end{pr}
\begin{proof}

    (i) Assume that the basis $\{|x,y\rangle\}$ is normalized, i.e., that
    \begin{equation}
        \langle x',y'\;|\;x'',y''\rangle = \delta (x'-x'') \delta (y'-y'').
    \end{equation}
    With equation \eqref{conv1} in mind, we first focus on the computation of $\langle x',k';b,c|x,y\rangle$. Recall now that $L_1$ and $X_3$ generate the centralizer $C_{X_1}(\mathfrak{J}_2)$ and that the two bases in play are formed of eigenvectors of these two operators in addition to $X_1$ which is diagonal in the two bases. One will recall also that $C_{X_1}(\mathfrak{J}_2)$ is a rank one Jacobi algebra centrally extended by $X_1$ with the identification of the standard generators $K_1$ and $K_2$ given in \eqref{X1central}. The overlaps of the eigenvectors of these two $\mathfrak{J}_1$ generators were seen to be given by univariate Jacobi polynomials (up to factors) in \eqref{overJac}. We have from this formula that the variable is the eigenvalue of $K_2$ and is here $\frac{y}{1-x}$ since
    \begin{equation}
        K_2\;|x,y\rangle = \left[\frac{X_3}{X_1 -1}+1\right]\;|x,y\rangle = \left[\frac{1-x-y}{x-1} +1\right]|x,y\rangle = \frac{y}{1-x}|x,y\rangle.
    \end{equation}
    The degree is determined by the eigenvalues of $K_1$ which in this case are those of $L_1$ on the vectors $|x,k;b,c\rangle$ that is, $-k(k+b+c)$, and one will recall the identification of the parameters also provided in \eqref{X1central}. Hence, the intermediate overlaps $\langle x',k';b,c|x,y\rangle$ are given by:
    \begin{equation}
       \langle x',k';b,c\;|\;x,y\rangle = \delta (x-x')\; C_{k'}(x,y)\; J_{k'}^{(b,c)}\left(\frac{y}{1-x} \right). \label{overint1}
    \end{equation}
    The presence of the delta function $\delta (x-x')$ follows from the fact that these overlaps involve two bases whose elements are both eigenvectors of $X_1$. The function $C_k(x,y)$ is determined by the condition that the basis $\{|x,k;b,c\rangle\}$ is orthonormalized, i.e., that
    \begin{equation}
        \langle x', k';b,c|x'',k'';b,c\rangle = \delta(x'-x'')\;\delta_{k',k''}.
    \end{equation}
    It is thus seen that
    \begin{equation}
        C_k(x,y) = \sqrt{\frac{1}{N_k^{(b,c)}}\frac{y^b(1-x-y)^c}{(1-x)^{b+c+1}}}, \label{Ck}
    \end{equation}
    with $N_k^{(b,c)}$ is given by \eqref{norm1}.
    Indeed, with $C_k(x,y)$ so defined, and using \eqref{overint1}
    \begin{align}
        &\langle x',k';b,c|x'',k'';b,c\rangle \nonumber \\
        &= \int_{0\le x\le  1-y \le 1} dx dy\; \langle x',k';b,c|x,y\rangle\langle x,y|x'',k'';b,c\rangle \nonumber \\
        &= \int_{0\le x\le  1-y \le 1}  dx dy\; \delta(x'-x)\delta(x''-x) C_{k'}(x',y)C_{k''}(x'',y) J_{k'}^{(b,c)}\left(\frac{y}{1-x}\right) J_{k''}^{(b,c)}\left(\frac{y}{1-x} \right)\nonumber \\
        &=\delta (x'-x'') \sqrt{\frac{1}{N_{k'}^{(b,c)}N_{k''}^{(b,c)}}}\int_0^{1-x'} dy\;\frac{y^b(1-x'-y)^c}{(1-x')^{b+c+1}} J_{k'}^{(b,c)}\left(\frac{y}{1-x'}\right) J_{k''}^{(b,c)}\left(\frac{y}{1-x'}\right) \nonumber \\
        &=\delta (x'-x'') \sqrt{\frac{1}{N_{k'}^{(b,c)}N_{k''}^{(b,c)}}}\int_0^{1-x'} \frac{dy}{(1-x')}\; \left[\left(\frac{y}{1-x'}\right)^b \left(1-\frac{y}{1-x'}\right)^c\right]J_{k'}^{(b,c)}\left(\frac{y}{1-x'}\right) J_{k''}^{(b,c)}\left(\frac{y}{1-x'}\right) \nonumber \\
        &=\delta(x'-x'')\;\delta_{k',k''},
    \end{align}
    from observing that the last integral above corresponds to the orthogonality relation \eqref{orth1var} of the univariate Jacobi polynomials. Note that the parameters must satisfy $b,c>-1$.

    \smallskip

    (ii) Turn now to the overlaps between $\{|n,k;a,b,c\rangle\}$ and $\{|x',k';b,c\rangle\}$ that are respectively eigenvectors of the elements $L$ and $X_1$  generating the centralizer $C_{L_1}(\mathfrak{J}_2)$ of $L_1$ which is diagonal on these two sets of vectors. It was seen in Subsection \ref{sub3.1}, that
    $C_{L_1}(\mathfrak{J}_2)$ is a centrally extended rank one Jacobi algebra and that the  canonical $\mathfrak{J}_1$ generator $K_1$ is obtained by shifting $L$ by a central term related to $L_1$ while $K_2$ is simply $X_1$. The identification of the parameters is $\alpha = a$ and $\beta=2k+b+c+1$ with respect to \eqref{rank1}. It follows that $\{|n,k;a,b,c\rangle\}$ and $\{|x',k';b,c\rangle\}$ are respectively eigenbases of $K_1$ and $K_2$ satisfying:
    \begin{align}
        K_1 \; |n,k;a,b,c\rangle\ =& \;[L+k(k+a+b+c+2)]\;|n,k;a,b,c\rangle \nonumber\\
        =& -(n-k)[(n-k)+a+b+c+2k+2] |n,k;a,b,c\rangle, \\
        K_2 \;|x',k';b,c\rangle\ =& \;x'\;|x',k';b,c\rangle.
    \end{align}
Given all this information, upon comparing with the general expression \eqref{overJac} for the overlaps between the eigenbases of $K_1$ and $K_2$, we conclude that the overlaps $\langle n,k;a,b,c|x',k';b,c\rangle$  will be Jacobi polynomials in the variale $x$ (the eigenvalue of $K_2$) of degree $n-k$ (in view of the eigenvalues of $K_1$) with parameters as given above. Note that this requires $n \ge k \ge 0$. We therefore have:
\begin{equation}
    \langle n,k;a,b,c\;|\;x',k';b,c\rangle= \delta _{k,k'} A_{n,k}(x')J_{n-k}^{(a, b+c+2k+1)}(x'), \label{overint2}
\end{equation}
where the function $A_{n,k}(x)$ is fixed by the requirement that the basis vectors $\{|n,k;a,b,c\rangle\}$ be orthonormalized, i.e., that they satisfy
\begin{equation}
    \langle m,\ell;a, b,c\;|\;n,k;a,b,c\rangle=\delta _{m,n} \delta _{\ell,k}.
\end{equation}
This implies that 
\begin{equation}
    A_{n,k}(x) = \sqrt{\frac{x^a (1-x)^{(b+c+2k+1)}}{N_{n-k}^{(a,b+c+2k+1)}}},\label{Ank}
\end{equation}
where $N_n^{(a,b)}$ is given by \eqref{norm1}. This is seen as follows. From \eqref{overint2} and \eqref{Ank} we have
\begin{align}
    &\langle m,\ell;a,b,c|n,k;a,b,c\rangle = \nonumber\\
    &\sum_{k'=0}^n \int_0^1 dx\;\langle m, \ell;a,b,c\;|\;x,k';b,c\rangle\langle \; x, k';b,c|\;n, k;a,b,c\rangle =\nonumber \\
    &\sum_{k'=0}^n \int_0^1 dx\;\delta _{\ell ,k'} \delta _{k ,k'} A_{m,\ell}(x) A_{n,k}(x) J_{m-\ell}^{(a,b+c+2\ell +1)}(x)J_{n-k}^{(a,b+c+2k +1)}(x) =\nonumber \\
    & \delta _{\ell ,k} \;\frac{1}{\sqrt{N_{m-k}^{(a,b+c+2k+1)} N_{n-k}^{(a,b+c+2k+1)}}}\int _0^1 dx \;x^a(1-x)^{(b+c+2k+1)} J_{m-k}^{(a,b+c+2k +1)}(x)J_{n-k}^{(a,b+c+2k +1)}(x) =\nonumber \\
    & \delta_{m,n} \delta_{\ell,k},
\end{align}
recognizing in the last step the orthogonality relation \eqref{orth1var} of the univariate Jacobi polynomials.

\smallskip

(iii) We now have all the ingredients needed to obtain the overlaps  $ \langle n,k;a,b,c |x,y\rangle $ from \eqref{conv1}. Using the expression for  $\langle n,k;a,b,c\;|\;x',k';b,c\rangle $ given by \eqref{overint2} and \eqref{Ank} and the one given by \eqref{overint1} and \eqref{Ck} for $\langle x',k';b,c\;|\;x,y\rangle$, equation  \eqref{conv1} yields
\begin{align}
      &\langle n,k;a,b,c|x,y\rangle = \nonumber \\
      &\sum_{k'=0}^n \int_0^1 dx'\;\delta _{k,k'} \sqrt{\frac{(x')^a (1-x')^{(b+c+2k+1)}}{N_{n-k}^{(a,b+c+2k+1)}}}J_{n-k}^{(a, b+c+2k+1)}(x') \\
 &\qquad\qquad\times\delta (x-x')\; \sqrt{\frac{1}{N_{k'}^{(b,c)}}\frac{y^b(1-x-y)^c}{(1-x)^{b+c+1}}}\; J_{k'}^{(b,c)}\left(\frac{y}{1-x} \right)= \nonumber \\
      & \sqrt{\frac{x^a y^b (1-x-y)^c}{N_k^{(b,c)}N_{n-k}^{(a,b+c+2k+1)}}} J_{n-k}^{(a,b+c+2k+1)}\bigl(x\bigr) \ (1-x)^k \ J_k^{(b,c)}\left(\frac{y}{1-x}\right).
      \end{align}
      This concludes the proof of Proposition \ref{prop 1} since the last three factors are recognized from \eqref{defJac2} to define the two-variable Jacobi polynomials $J_{n,k}^{(a,b,c)}(x,y)$ .
\end{proof}
\subsection{Alternate two-variable Jacobi polynomials}
As mentioned above, it is also natural to consider the expansion of the basis vectors $|x,y\rangle$ associated with the top vertex of the pentagon in Figure \ref{fig:tr}, in terms of the vectors $|n,k;a,b,c\rangle ^{\pi}$ defined in \eqref{brt}, which form the basis corresponding to the bottom right corner of the pentagon. The computation of these expansion coefficients will follow closely the approach adopted in the preceding subsection to obtain the overlaps between the basis vectors associated with the apex and the bottom left corner of the pentagon. Consider here yet another basis formed by the elements $\{|(1-x-y), k;a,b\rangle^{\pi}\}$ that are the joint eigenvectors of $X_3$ and $L_3$:
\begin{align}
    X_3\;|(1-x-y), k; a,b\rangle^{\pi}&=(1-x-y)\;|(1-x-y), k; a,b\rangle^{\pi}, \\
    L_3\;|(1-x-y) ,k; a,b\rangle^{\pi}&=-k(k+a+b+1)\;|(1-x-y),k; a,b\rangle^{\pi}.
\end{align}
We shall enforce the orthonormalization of these vectors and use the resolution of the identity that they entail to obtain the overlaps ${}^{\scriptstyle \pi}\langle n,k;a,b,c\;|\;x,y\rangle$ using the formula
\begin{align}
    &{}^{\scriptstyle \pi}\langle n,k;a,b,c \mid x,y\rangle = \nonumber \\
    &\quad \sum_{k'} \int_0^1 \! d(1 - x' - y') \; 
    {}^{\scriptstyle \pi}\langle n,k;a,b,c \mid (1 - x' - y'), k';a,b\rangle ^{\pi}\,
    {}^{\scriptstyle \pi}\langle (1 - x' - y'), k';a,b \mid x, y\rangle. \label{conv2}
\end{align}
This amounts to convoluting with the intermediate basis vectors along right hand-side path from top to bottom along the edges of the pentagon. This portion of Figure \ref{fig:tr} is represented below in Figure \ref{fig:3} which depicts the algebraic underpinning of the alternate two-variable polynomials.
\begin{figure}[ht] 
\centering
\begin{tikzpicture}[scale=1.5]
  \coordinate (A) at (0,0);
  \coordinate (B) at (2,0);
  \coordinate (C) at (4,0);

  \fill (A) circle (2pt);
  \fill (B) circle (2pt);
  \fill (C) circle (2pt);

  \draw (A) -- (B) -- (C);

  \node at ($(A)+(0,0.4)$) {\(X_3,\,X_1\)};
  \node at ($(B)+(0,0.4)$) {\(L_3,\,X_3\)};
  \node at ($(C)+(0,0.4)$) {\(L,\,L_3\)};
\end{tikzpicture}
\caption{The right hand side of the pentagon of Figure \ref{fig:tr} that connects the $\{|x,y\rangle\}$ basis to the $\{|n,k;a,b,c\rangle^\pi\}$ basis via the $\{|(1-x-y),k;a,b\rangle^\pi\}$ basis.\label{fig:3}}
\end{figure}
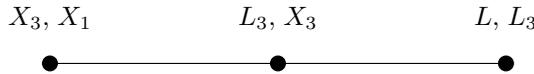

\begin{pr}\label{pro:2}
     Let $\{|x,y\rangle |0\le x\le 1,0\le y\le 1-x\}$ and $\{|n,k;a,b,c\rangle^{\pi}|n,k\in\mathbb{N}, 0\le k\le n\}$ be the two representation bases of $\mathfrak{J}_2$ defined by equations \eqref{u} and \eqref{brt}. The overlaps between the vectors of these two bases are given by
     \begin{equation}
         {}^{\scriptstyle \pi}\langle n,k;a,b,c \mid x,y\rangle = \sqrt{\frac{x^a y^b (1-x-y)^c}{N_k^{(b,a)}N_{n-k}^{(c,a+b+2k+1)}}}J_{n,k}^{(c,b,a)}\left[(1-x-y),y\right].\label{eqprop2}
     \end{equation}
     We observe that the expression for this overlap is obtained from the overlap $\langle n,k;a,b,c|x,y\rangle$ given in Proposition \ref{prop 1} by performing simultaneously the permutations:
     \begin{equation}
         x \leftrightarrow 1-x-y, \qquad a \leftrightarrow c.
     \end{equation}
\end{pr}
\begin{proof}
    (i) We start by computing the overlaps between the eigenvectors $|\; (1 - x' - y'), k';a,b \rangle ^{\pi}$ of $L_3$ and the eigenvectors $|x,y\rangle$ of $X_1$ that generate the centralizer $C_{X_3}(\mathfrak{J}_2)$ of $X_3$. This algebra was recognized in Subsection \ref{subs:X3} as a centrally extended rank one Jacobi algebra with parameters $\alpha = b,\; \beta = a$. The canonical generators $K_1$ and $K_2$ identified in \eqref{JacsubL3X1} are each diagonal in one of these two bases with the eigenvalues of $K_1=L_3$ equal to $-k(k+a+b+1)$ in the first basis and those of $K_2$ given by: 
    \begin{equation}
        K_2 \;|\;x, y\rangle = \left[\frac{X_1}{X_3 -1} +1\right]\;|\;x, y\rangle=\left(\frac{y}{x+y}\right)\;|\;x, y\rangle,
    \end{equation}
    in the second basis.
    Comparison with the expression \eqref{overJac} for the overlaps between eigenvectors of the $\mathfrak{J}_1$-generators $K_1$ and $K_2$, tells us that $  {}^{\scriptstyle \pi}\langle (1 - x' - y'), k';a,b \;| \;x,y\rangle$ will involve a univariate Jacobi polynomial in the variable given by the eigenvalues of $K_2$ which is of degree $k$ in light of the eigenvalues of $K_1$. Given the correspondence of the parameters, one precisely has:
\begin{equation}
     {}^{\scriptstyle \pi}\langle (1 - x' - y'), k';a,b \;| \;x,y\rangle = \delta (x'+y'-x-y) D_{k'}\left[(x+y), y\right] J_{k'}^{(b,a)}\left(\frac{y}{x+y}\right),
\end{equation}
    where
    \begin{equation}
        D_{k'}\left[(x+y), y\right]=\sqrt{\frac{1}{N_{k'}^{(b,a)}} \frac{x^ay^b}{(x+y)^{a+b+1}}}.
    \end{equation}
The $\delta$-function occurs
because the elements of both bases are eigenvectors of the central $X_3$. 
The function $D_{k'}\left[(x+y), y\right]$ is fixed by the condition that the basis vectors $|\;(1 - x' - y'), k';a,b\rangle ^{\pi}$ verify the orthonormalization condition
\begin{equation}
         {}^{\scriptstyle \pi}\langle (1 - x' - y'), k';a,b \;|\;(1 - x'' - y''), k'';a,b\rangle ^{\pi} = \delta _{k',k''} \delta (x'+y'-x''-y'').
\end{equation}
This is checked as was done in Subsection \ref{subs:1stJ}.

\smallskip

(ii) With an eye to \eqref{conv2}, the next step is to compute the overlaps between the vectors $\{|\;n,k;a,b,c\rangle^{\pi}\}$ and those of the set $\{|\;(1 - x' - y'), k';a,b\rangle ^{\pi}\}$ which are respectively eigenvectors of $L$ and $X_3$ while both diagonalizing $L_3$. They form two representation bases of the centralizer $C_{L_3}(\mathfrak{J}_2)$ of $L_3$ which is again a central extension of the rank one Jacobi algebra. The correspondence is carried out in Subsection \ref{subs:L3}. The upshot is on the one hand, that the parameters of the presentation \eqref{rank1} are $\alpha =c$ and $\beta=a+b+2k+1$ and on the other hand, that the family $\{|(1 - x' - y'), k';a,b\rangle ^{\pi}\}$ is an eigenbasis of the standard generator $K_2$ since it is identified as $X_3$ while the vectors $|n,k;a,b,c\rangle^{\pi}\} $ form an eigenbasis of $K_1$ satisfying
\begin{align}
    K_1 \;|n,k;a,b,c\rangle^{\pi} &= \left[L + k(k + a + b + c + 2)\right] \;|n,k;a,b,c\rangle ^{\pi}\nonumber \\
    &= (n - k)\left[(n - k) - a + b + c + 2k + 2\right] \;|n,k;a,b,c\rangle^{\pi}.
\end{align}
From formula \eqref{overJac} that gives the expression of the overlaps between the eigenbases of the $\mathfrak{J}_1$-generators $K_1$ and $K_2$, we have that 
\begin{equation}
    {}^{\scriptstyle \pi}\langle n,k;a,b,c \;| (1-x'-y'),k';a,b\rangle ^{\pi} = \delta_{k,k'} \;B_{n,k}(1-x'-y') \;J_{n-k}^{c, a+b+2k+1)} (1-x'-y'), \label{pi-int}
\end{equation}
with 
\begin{equation}
    B_{n,k}(1-x-y) =\sqrt{\frac{(1-x-y)^c (x+y)^{a+b+2k+1}}{N_{n-k}^{(c,a+b+2k+1)}}},
\end{equation}
so that 
\begin{align}
     &{}^{\scriptstyle \pi}\langle n,k;a,b,c \;|\;n,k'';a, b, c\rangle  \nonumber \\
     &=\;\sum_{k'=0}^{n}\; \int_0 ^1\;d(1-x'-y') \;{}^{\scriptstyle \pi}\langle n,k;a,b,c \;|(1-x'-y'), k';a,b\rangle ^{\pi}\;{}^{\scriptstyle \pi}\langle (1-x'-y'),k';a,b \;|n, k'';a,b,c\rangle ^{\pi}  \nonumber \\
     & = \;\delta_{k,k''}.
\end{align}
This is validated in a by now routine way and the presence of the Kronecker delta is due to the fact that $L_3$ is central. Note again, in view of \eqref{pi-int} now, that the range of $n$ and $k$ over $\mathbb{N}$ must be restricted: $n \ge k \ge 0$.

\smallskip
(iii) Substituting the results of these computations in \eqref{conv2}, one arrives at the formula:
\begin{align}
        &{}^{\scriptstyle \pi}\langle n,k;a,b,c \mid x,y\rangle = \nonumber \\
    &\sqrt{\frac{x^a y^b(1-x-y)^c}{N_k^{(b,a)}N_{n-k}^{(c,a+b+2k+1)}}}\; J_{n-k}^{(c, a+b+2k+1)}(1-x-y)\;(x+y)^{2k}\;J_k^{(b,a)}\;\left(\frac{y}{x+y}\right),
\end{align}
which confirms equation \eqref{eqprop2}. Furthermore, it is readily observed that 
\begin{equation}
      {}^{\scriptstyle \pi}\langle n,k;a,b,c \mid x,y\rangle =\langle n,k;c,b,a\;|\;(1-x-y), y\rangle ,
\end{equation}
thus concluding the proof of Proposition \ref{pro:2}.
\end{proof}

Guided by Figure \ref{fig:tr}, we shall next show that the elements of the family of two-variable Jacobi polynomials attached to one of the two top-to-bottom paths of the pentagon have expansion coefficients over the elements of the family associated to the  other path that are given in terms of Racah polynomials.

\section{Overlaps and Racah polynomials \label{sec:over}}

We have focused up to now on the overlaps $\langle n,k;a,b,c\;|\;x,y\rangle$ and  ${}^{\scriptstyle \pi}\langle n,k;a,b,c\mid x,y\rangle$ between the bases attached to the top and the two bottom vertices of the pentagon of Figure \ref{fig:tr}. As stated in Propositions \ref{prop 1} and \ref{pro:2}, these were found to be expressed in terms of the two-variable Jacobi polynomials $J_{n,k}^{(a,b,c)}(x,y)$ and their transform $J_{n,k}^{(c,b,a)}(1-x-y,y)$ under the transpositions: $a \leftrightarrow c$ and $x \leftrightarrow 1-x-y$. In order to determine the relations between these sets of functions, let us now pay attention to the connection between the bases attached to the left and right bottom vertices of the pentagon of Figure \ref{fig:tr} whose overlaps will read ${}^{\scriptstyle \pi}\langle n,\ell;a,b,c\;|\;n,m;a,b,c\rangle$. 

We recall (see \eqref{blt} and \eqref{brt}) that the vectors $|\;n, m; a, b, c\;\rangle$ are eigenvectors of $L_1$ with eigenvalues $-m(m+b+c+1)$, that the vectors $|\;n, \ell; a, b, c\;\rangle ^{\pi}$ are eigenvectors of $L_3$ with eigenvalues $-\ell(\ell + a+b+1)$ and also, that both sets are eigenvectors of $L$ with eigenvalues $-n(n+a+b+c+2)$.

We observed in Subsection \ref{subs:L} that $L_1$ and $L_3$ generate the rank one Racah algebra $\mathfrak{R}_1$ via the identification $K_1=L_1$, $K_2=L_3$ of the standard generators and with the coefficients $\xi, \eta_1, \eta_2, \zeta_1, \zeta_2$ of the defining relations \eqref{Ract} given respectively by \eqref{xiJac}, \eqref{eta1Jac}, \eqref{eta2Jac}, \eqref{zeta1Jac}, \eqref{zeta2Jac}.

We indicated in Subsubsection \ref{subs:relR} that the overlaps between the elements of the two representation bases of $\mathfrak{R}_1$ defined as eigenvectors of $K_1$ on the one hand and of $K_2$ on the other, are expressed in terms of Racah polynomials. The explicit formula is given in equation \eqref{ove_R} and the correspondence between the parameters $\xi, \eta_1, \eta_2, \zeta_1, \zeta_2$ of the $\mathfrak{R}_1$ relations and the parameters $\alpha, \beta, \gamma, \delta$ of the Racah polynomials $R_n(\lambda(\ell); \alpha, \beta, \gamma, \delta)$ given by the relations \eqref{xiRac}, \eqref{eta1Rac}, \eqref{eta2Rac}, \eqref{zeta1Rac}, \eqref{zeta2Rac}.

All this can now be used to give the overlaps between the bases attached to the two bottom vertices of the pentagon of Figure \ref{fig:tr}. From the relations between the parameters $a, b, c$ of the rank two Jacobi algebra and those $\xi, \eta_1, \eta_2, \zeta_1, \zeta_2$ of $\mathfrak{R}_1$ as realized by $L_1$ and $L_3$, and from the formulas connecting the parameters of the algebra $\mathfrak{R}_1$ to the $\alpha, \beta, \gamma, \delta$ of the Racah polynomials, one can express the latter in terms of $a, b, c$. This is done by solving the equations provided by equating the formulas for  $\xi, \eta_1, \eta_2, \zeta_1, \zeta_2$ given in the set of equations \eqref{xiJac}, \eqref{eta1Jac}, \eqref{eta2Jac}, \eqref{zeta1Jac}, \eqref{zeta2Jac} to those given in the set \eqref{xiRac}, \eqref{eta1Rac}, \eqref{eta2Rac}, \eqref{zeta1Rac}, \eqref{zeta2Rac}. Recalling that the central $L$ takes the value $-n(n+a+b+c+2)$, this gives \footnote{We may remark that other solutions such as $\alpha =b,\; \beta=c, \;\gamma=a+b+c+n+1, \;\delta=-c-n-1$, yield the same Racah polynomials $R_n(\lambda(x); \alpha, \beta, \gamma, \delta)$ as seen from their definition \eqref{defRac}.}
\begin{equation}
    \alpha =b, \qquad \beta=c, \qquad \gamma=-n-1, \qquad \delta = n+1+a+b.
\end{equation}
We may hence infer from \eqref{ove_R} that 
\begin{align}
    {}^{\scriptstyle \pi}\langle n,\ell;a,b,c\;|\;n,m;a,b,c\rangle &= (-1)^m\;\sqrt{\frac{w^{(b,c,-n-1,n+1+a+b)}(\ell)}{M_m^{(b,c,-n-1,n+1+a+b)}}}R_m\left(\lambda^{(a,b)}(\ell); b,c,-n-1, n+1+a+b\right) \nonumber \\
    &=(-1)^m\;S_m^{(b, c, -n-1, n+1+a+b)}(\ell),
    \label{31over}
\end{align} 
with $w^{(\alpha, \beta, \gamma, \delta)}(x)$ given by \eqref{weight} and $M_n^{(\alpha, \beta, \gamma, \delta)}$ by \eqref{normal} and \eqref{K} and where
\begin{equation}
    \lambda^{(a,b)} (\ell)= \ell(\ell + a+b+1).
\end{equation}

\noindent Note the symmetrical form of the intervening Racah polynomials:
\begin{align}
&R_m\left(\lambda^{(a,b)}(\ell); b, c, -n-1, n+1+a+b\right) =\nonumber  \\
&{}_4F_3\left( \begin{array}{c}
-m,\ m+b+c+1,\ -\ell,\ \ell+a+b+1 \\
b+1,\ n+a+b+c+2,\ -n
\end{array} ; 1 \right), \quad m, \ell = 0, 1, 2, \ldots, n,
\end{align}
which is such that
\begin{equation}
    R_m\left(\lambda^{(a,b)}(\ell); b, c, -n-1, n+1+a+b\right)=R_{\ell}\left(\lambda^{(c,b)}(m); b, a, -n-1, n+1+b+c\right).
\end{equation}
It is checked also, using say \eqref{orthfn}, that the orthonormalized functions $ S_m^{(\alpha, \beta, \gamma, \delta)}(\ell)$ admit the same symmetry, i.e.,
\begin{equation}
     S_m^{(b, c, -n-1, n+1+a+b)}(\ell) =  S_{\ell}^{(b, a, -n-1, n+1+b+c)}(m) \label{duality},
\end{equation}
under the exchanges
\begin{equation}
    m \; \leftrightarrow \;\ell, \qquad a\; \leftrightarrow \;c.
\end{equation}
\noindent From
\begin{equation}
   {}^{\scriptstyle \pi}\langle\; n,\ell;a,b,c\;|\;x, y\;\rangle=
   \sum_{m=0}^{n} {}^{\scriptstyle \pi}\langle \; n,\ell;a,b,c\;|\;n, m; a, b, c\;\rangle \langle \;n, m;a, b, c\;|\;x, y \rangle, 
\end{equation}
equation \eqref{31over}, and the results of Propositions \ref{prop 1} and \ref{pro:2}, one arrives at
\begin{pr}\label{pro:3}
The expansion of the bivariate Jacobi polynomials $J_{n,k}^{(c,b,a)}\left[(1-x-y),y\right]$  in terms of the two-variable Jacobi polynomials  $J_{n,k}^{(a, b, c)} (x, y)$ with $(n \ge k \ge 0)$ is given as follows:
\begin{align}
    &J_{n,\ell}^{(c,b,a)}\left[(1-x-y),y\right] = \noindent \\
    &\mathfrak{F}^{(c,b,a;\;n)}(\ell) \;  \sum_{m=0}^{n} \; \frac{(-1)^m}{\mathfrak{G}_m ^{(a,b,c;\;n)}} \; R_m\left(\lambda^{(a,b)}(\ell); b,c,-n-1, n+1+a+b\right) \; J_{n,m}^{(a, b, c)} (x, y) ,\label{Jpiexp}
\end{align}
where $R_n(\lambda(x); \alpha, \beta, \gamma, \delta)$ are univariate Racah polynomials and where
\begin{equation}
    \mathfrak{F}^{(c,b,a;\;n)}(\ell)= \sqrt{w^{(b,c,-n-1,n+1+a+b)}(\ell)\; N_{\ell} ^{(b,a)}\; N_{n-\ell}^{(c,a+b+2\ell +1)}},
\end{equation}
and 
\begin{equation}
    \mathfrak{G}_m ^{(a,b,c;\;n)} = \sqrt{M_m^{(b,c,-n-1,n+1+a+b)}\; N_m^{(b,c)}\; N_{n-m}^{(a, b+c+2m+1)}}.
\end{equation}
\end{pr}
This result was obtained by Dunkl in \cite{dunkl1984orthogonal} (with different normalizations). See also \cite{labriet2024realisations}. The presence of the $(-1)^m$ factor will be justified in Section \ref{sec:symm}.
Keeping in mind that the bivariate Jacobi polynomials are composed of two single-variable Jacobi polynomials, equation \eqref{Jpiexp} corresponds to the Biedenharn-Elliott \cite{biedenharn1952identity}, \cite{elliott1953theoretical} formula for Racah coefficients in the continuum limit.

\noindent From the orthogonality of the bivariate Jacobi polynomials or equivalently from
\begin{equation}
   {}^{\scriptstyle \pi}\langle n,\ell;a,b,c\;|\;n,m;a,b,c\rangle =  \int_{0\le x\le  1-y \le 1}  dx \; dy  {}^{\scriptstyle \pi}\langle n,\ell;a,b,c\;|\;x,y\;\rangle \langle x,y\;|\;n,m;a,b,c\rangle,
\end{equation}
we also see that the Racah polynomials can be written as the convolution of two bivariate Jacobi polynomials according to
\begin{align}
    &\mathfrak{F}^{(c,b,a;\;n)}(\ell) \; R_m\left(\lambda^{(a,b)}(\ell); b,c,-n-1, n+1+a+b\right) = \nonumber \\
    & (-1)^m \sqrt{\frac{M_m^{(b,c,-n-1,n+1+a+b)}}{ N_m^{(b,c)}\; N_{n-m}^{(a, b+c+2m+1)}}}\; \int_{0\le x\le  1-y \le 1}  dx\;dy\;x^a y^b (1-x-y)^c\;
    J_{n,\ell}^{(c,b,a)}\left[(1-x-y),y\right]\; J_{n,m}^{(a,b,c)}(x,y).
\end{align}

We have been providing an algebraic interpretation of the two-variable Jacobi polynomials on the triangle. As pointed out in \cite{dunkl1984orthogonal} and \cite{dunkl2014orthogonal}, these functions and their description should exhibit the order three symmetry of that simplex. Two families of bivariate Jacobi polynomials related to each other by one such transformation appeared so far and this symmetry was further found to be represented in this Jacobi polynomial basis by a matrix involving the Racah polynomials in its entries. The fuller expression of the underlying order three-fold symmetry is addressed in the next section. 

\section{Symmetry of order three \label{sec:symm}}

In his discussion of the symmetry of order three of the bivariate Jacobi polynomials, Dunkl introduced an additional family, namely, $J_{n,k}^{(b,a,c)}(y,x)$ which is obtained from the standard $J_{n,k}^{(a,b,c)}(x,y)$ by the cycle $x\leftrightarrow y$ and $ a \leftrightarrow b$ which is complementary to the one leading to $J_{n,k}^{(c,b,a)}\left[(1-x-y),y\right]$. This is motivated by the nature of the underlying simplex. We next examine how this supplementary family can be integrated into the framework developed so far.

\subsection{The family $J_{n,k}^{(b,a,c)}(y,x)$ \label{sec:bac}}

The issue here is to design another pentagon diagram as in Figure \ref{fig:tr}, such that one of its top-to-bottom paths again leads to the standard two-variable Jacobi polynomials and the overlaps between the bases on the other side are the polynomials in the title of this subsection. If the two bases attached to the bottom edge of the pentagon have Racah polynomials as mutual expansion coefficients in this case also, this would show that the algebraic interpretation of the two-variable Jacobi polynomials offered here properly integrates the order three symmetry of the simplex on which these polynomials are defined. This diagram is provided in Figure \ref{fig:4} and we shall now explain that it has the desired features.

As hinted by the notation that had the generators of the rank two Jacobi algebra $\mathfrak{J}_2$ labelled asymmetrically by the indices $1$ and $3$,
to achieve the goal stated above we need to modify the set of generators and to introduce correspondingly new eigenbases. To this end we shall replace $L_3$ and $X_3$ in the set of generators by $L_2$ and $X_2$ which are defined by
\begin{equation}
    L_2=L - L_1 - L_3, \qquad \text{and} \qquad X_2=1-X_1-X_3. \label{def2}
\end{equation}
This is what these symbols stand for in the pentagonal diagram of Figure \ref{fig:4} below. 

First, it is seen that $[L_2, X_2] = 0$. Indeed, one readily finds that
\begin{equation}
    [L_2, X_2]=-[L,X_1]-[L,X_3]+[L_1,X_3]+[L_3,X_1],
\end{equation}
which is shown to be identically $0$ in Appendix \ref{APP A}, see equation \eqref{identity}. It is moreover obvious that $[L,L_2]=0$ since $L$ commutes with $L_1$ and $L_3$. Hence, the pairs attached to the vertices of the pentagon of Figure \ref{fig:4} are all commutative.

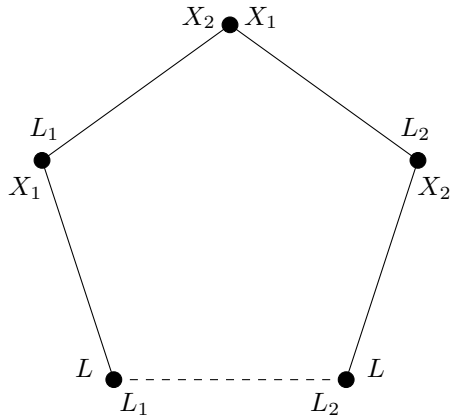
\begin{figure}[htbp] 
\begin{center}
\begin{minipage}[t]{.9\linewidth}
\begin{center}
\begin{tikzpicture}[scale=1.3]
   
    \def\R{2cm}
    \def\Rt{2.1cm}
    \def\Rtt{1.85cm}

    \foreach \i in {1,...,5} {
        \coordinate (P\i) at ({90 + (\i-1)*360/5}:\R);
    }
    \draw (P3)--(P2) -- (P1) -- (P5) -- (P4);
    \draw[dashed] (P3) -- (P4);


\node at ({90 + (+0.12)*360/5}:\Rt) {$X_2$};
\node at ({90 + (-0.12)*360/5}:\Rt) {$X_1$};
\node at ({90 + (-0.12+1)*360/5}:\Rt) {$L_1$};
\node at ({90 + (+0.12+1)*360/5}:\Rt) {$X_1$};
\node at ({90 + (-0.12-1)*360/5}:\Rt) {$X_2$};
\node at ({90 + (+0.12-1)*360/5}:\Rt) {$L_2$};
\node at ({90 + (-0.12-2)*360/5}:\Rt) {$L_2$};
\node at ({90 + (+0.12-2)*360/5}:\Rt) {$L$};
\node at ({90 + (-0.12+2)*360/5}:\Rt) {$L$};
\node at ({90 + (+0.12+2)*360/5}:\Rt) {$L_1$};

     \foreach \i in {1,...,5} {
        \draw[fill]  (P\i) circle (0.08); ;
     }

\end{tikzpicture}
\caption{This pentagon has the two-variable Jacobi polynomials $J_n^{(a,b,c)}(x,y)$ associated, like in Figure \ref{fig:tr}, to the left-hand path from top to bottom.  The polynomials corresponding to the right-hand path are the polynomials obtained from the latter under the simultaleous exchanges $x \leftrightarrow y$ and $ a \leftrightarrow b$, i.e $J_n^{(b,a,c)}(y,x)$. This diagram depicts the connection between the two families of bivariate Jacobi polynomials that are related under a permutation complementary to the one arising on Figure \ref{fig:tr}, namely $x \leftrightarrow 1-x-y$ and $ a \leftrightarrow c$.\label{fig:4} }

\end{center}
\end{minipage}
\end{center}
\end{figure}

\noindent These pairs of generators will again be used to define eigenbases for the $\mathfrak{J}_2$ module under consideration. Note that $X_2$ acts as multiplication by $y$ on the basis vectors $|\;x,y\rangle$:
\begin{equation}
    X_2 \;|\;x,y\rangle=[1-x-(1-x-y)]\;|\;x,y\rangle\;=y\;|\;x,y\rangle.
\end{equation}
\subsubsection{The left-hand side of Figure \ref{fig:4}}

The bases attached to the left vertices of Figure \ref{fig:4} are the same as those associated to the corresponding vertices of Figure \ref{fig:tr}. Hence the computation of the overlaps  $\langle n,k;a,b,c\;|\;x,y\rangle$ between the top and bottom left bases should not depend on which diagram is used and should give the result of Proposition \ref{prop 1}. Yet, in Figure \ref{fig:4}, $X_2$ instead of $X_3$ appears at the summit, with the effect, at least superficially, of modifying the realization of the rank one Jacobi algebra associated to the first edge from the top. Let us examine how this does not change the end outcome for the overlaps.

Given the definition \eqref{def2} of $X_2$, one observes with the help of the defining relations of $\mathfrak{J}_2$ given in Appendix \ref{APP A} that $L_1$ and $X_2$ generate a subalgebra with the following commutation relations:
   \begin{subequations}\label{L1X2L1tt}
         \begin{align}   
 [[L_1,X_2],L_1]&=2\{X_2 ,L_1\}+\{X_1,L_1\}-2L_1  
 -(b+c) \left((b+c+2)X_2 -(b+1)(I-X_1) \right),\label{L1X2L1t}\\
 [[L_1,X_2],X_2]&=-2X_2^2+2(I-X_1)X_2.\label{L1X2X2t}
 \end{align}
   \end{subequations}
Direct comparison shows that these relations have precisely the same structure as the relations \eqref{L1X3L1tt} between $L_1$ and $X_3$ and are obtained from the latter under the replacements:
\begin{equation}
    X_3 \rightarrow X_2; \qquad c \rightarrow b, \quad b\rightarrow c.
\end{equation}
From the analysis of the centralizer $\mathfrak{C}_{X_1}(\mathfrak{J}_1)$ performed in Subsection \ref{subs:X1} for which we are here using instead $L_1$ and $X_2$ as generators, we infer by analogy that these two elements form a rank one Jacobi algebra \eqref{rank1} where the canonical generators $K_1$ and $K_2$, and the parameters $\alpha$ and $\beta$ are identified as follows:
\begin{equation}
    K_1=L_1, \qquad K_2= 1-\frac{X_2}{1-X_1}; \qquad \alpha=c, \quad \beta=b. \label{ident2l}
\end{equation}
Besides, the same $\{|x,k;b,c\rangle\}$ form an eigenbasis of $K_1$, (see \eqref{intbas}) with eigenvalues $-k(k+b+c)$ and the set $\{|x,y\rangle\}$ is an eigenbasis of $K_2$ satisfying:
\begin{equation}
    K_2\;|x,y\rangle= \left(1-\frac{y}{1-x}\right)\;|x,y\rangle,
\end{equation}
in light of \eqref{ident2l}. It follows according to the property \eqref{overJac} of the algebra $\mathfrak{J}_1$ we used many times, that the overlaps between these two eigenbases are given in this context by Jacobi polynomials with parameters $(\alpha, \beta)=(c,b)$ and variable given by the eigenvalues of $K_2$, precisely:
    \begin{equation}
       \langle x',k';b,c\;|\;x,y\rangle = \delta (x-x')\; \Tilde{C}_{k'}(x,y)\; J_{k'}^{(c,b)}\left(1-\frac{y}{1-x} \right), \label{overint1alt}
    \end{equation}
where $\Tilde{C}_k(x,y)$ would be chosen to ensure the orthonormalization of the basis $\{|\;x,k;b,c\rangle\}$. At first glance the expression above does not appear to coincide with the formula \eqref{overint1} for $ \langle x',k';b,c\;|\;x,y\rangle$. However, in light of definition \eqref{defJac1} and of the Kummer transformation formula \eqref{Kummer}, it is checked that
\begin{equation}
    J_n^{(a,b)}(1-x) = (-1)^n \; J_n^{(b,a)}(x). \label{1-x}
\end{equation}
In view of this observation, it is readily seen that equation \eqref{overint1alt} is turned into \eqref{overint1} under the change of variable $1-\frac{y}{1-x} \rightarrow \frac{y}{1-x }$ (since the choice of the normalizing factor must follow). As expected, the left-hand side of the pentagon of Figure \ref{fig:4}, will lead in the end, to the standard two-variable Jacobi polynomials $J_{n,k}^{(a,b,c)}(x,y)$. With this settled, there remains to see if the right-hand side corresponds to the polynomials $J_{n,k}^{(b,a,c)}(y,x)$. We turn to that question now using the same approach as before and hence proceeding swiftly.

\subsubsection{The right-hand side of Figure \ref{fig:4}}

New orthonormalized bases are defined as joint eigenvectors of the commuting elements attached to the vertices on the right side of the pentagon of Figure \ref{fig:4}, namely
\begin{enumerate}[leftmargin=*]
    \item The eigenvectors $|\;n,k;a,b,c\rangle^\sigma$ of the bottom (right) corner are defined by:
    \begin{subequations}\label{sigmabt}
          \begin{align} 
        &L\;|\;n,k;a,b,c\rangle^\sigma=-n(n+a+b+c+2)\;|\;n,k;a,b,c\rangle^\sigma, \label{sigmab1}\\
        &L_2\;|\;n,k;a,b,c\rangle^\sigma=-k(k+a+c+1)\;|\;n,k;a,b,c\rangle^\sigma.\label{sigmab2}
    \end{align}
    \end{subequations}
    \item The eigenvectors $|\;y,k;a,c\rangle^\sigma$ of the middle (right) vertex are defined by the eigenvalue equations:
    \begin{subequations}\label{sigmamt}
          \begin{align}
          &L_2\;|\;y,k;a,c\rangle^\sigma=-k(k+a+c+1)\;|\;y,k;a,c\rangle^\sigma, \label{sigmam1}\\
          &X_2\;|\;y,k;a,c\rangle^\sigma=y\;|\;y,k;a,c\rangle^\sigma. \label{sigmam2}
    \end{align}
    \end{subequations}
\end{enumerate}
We want to obtain the overlaps ${}^{\scriptstyle \sigma}\langle n,k;a,b,c \mid x,y\rangle$ between the top and bottom basis vectors by using the resolution of identity involving the middle basis, i.e.,
\begin{align}
    &{}^{\scriptstyle \sigma}\langle n,k;a,b,c \mid x,y\rangle = \nonumber \\
    &\quad \sum_{k'} \int_0^1 \! dy' \; 
    {}^{\scriptstyle \sigma}\langle n,k;a,b,c \mid  y', k';a,c\rangle ^{\sigma}\,
    {}^{\scriptstyle \sigma}\langle y', k';a,c \mid x, y\rangle. \label{conv3}
\end{align}
This computation will amount, in this case, to determining the overlaps between the $\mathfrak{J}_1$-bases defined by the eigenvectors of the generators at the boundaries of the segments of Figure \ref{fig:5} below which reproduces the right-hand side of the pentagon of Figure \ref{fig:4}.
\begin{figure}[ht] 
\centering
\begin{tikzpicture}[scale=1.5]
  \coordinate (A) at (0,0);
  \coordinate (B) at (2,0);
  \coordinate (C) at (4,0);

  \fill (A) circle (2pt);
  \fill (B) circle (2pt);
  \fill (C) circle (2pt);

  \draw (A) -- (B) -- (C);

  \node at ($(A)+(0,0.4)$) {\(X_2,\,X_1\)};
  \node at ($(B)+(0,0.4)$) {\(L_2,\,X_2\)};
  \node at ($(C)+(0,0.4)$) {\(L,\,L_2\)};
\end{tikzpicture}
\caption{The right hand side of the pentagon of Figure \ref{fig:4} that connects the $\{|x,y\rangle\}$ basis to the $\{|n,k;a,b,c\rangle^\sigma\}$ basis via the $\{|(y,k;a,c\rangle^\sigma\}$ basis.\label{fig:5}}
\end{figure}
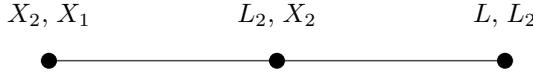
\begin{pr}\label{pro:4}
     Let $\{|x,y\rangle| 0\le x\le 1,0\le y\le 1-x\}$ and $\{|n,k;a,b,c\rangle^{\sigma}|n,k\in\mathbb{N}, 0\le k\le n\}$ be the two representation bases of $\mathfrak{J}_2$ defined by equations \eqref{u} and \eqref{sigmabt}. The overlaps between the vectors of these two bases are given by
     \begin{equation}
         {}^{\scriptstyle \sigma}\langle n,k;a,b,c \mid x,y\rangle = \sqrt{\frac{x^a y^b (1-x-y)^c}{N_k^{(a,c)}N_{n-k}^{(b,a+c+2k+1)}}}J_{n,k}^{(b,a,c)}(y,x). \label{eqprop4}
     \end{equation}
     The expression for this overlap is obtained from the overlap $\langle n,k;a,b,c\;|x,y\rangle$ given in Proposition \ref{prop 1} by performing simultaneously the permutations:
     \begin{equation}
         x \leftrightarrow y, \qquad a \leftrightarrow b.
     \end{equation}
\end{pr}
This proposition establishes the result we set out to obtain so as to provide an algebraic interpretation of the two-variable Jacobi polynomials that encompasses the symmetries of the triangle and includes the three families of these functions introduced by Dunkl \cite{dunkl1984orthogonal}. Let us now complete its derivation.
\begin{proof}
(i) In order to obtain the overlaps between the eigenvectors $|x,y\rangle$ of $X_1$ and those, $|y,k;a,c\rangle^\sigma$ of $L_2$, we need to determine the algebra that these generators realize knowing that they centralize $X_2$. Recalling that $L_2=L-L_1-L_3$ (see \eqref{def2}), one finds from the defining relations of $\mathfrak{J}_2$ given in Appendix \ref{APP A} that $X_1$ and $L_2$ verify:
\begin{subequations}\label{L2X1L2tt}
    \begin{align}
    [[L_2,X_1],L_2]&=\{2X_1+X_2-I,L_2\} 
 -(a+c)\left((a+1)(X_1+X_2-I)+(c+1)X_1\right), \label{L2X1L2t}\\
    [[L_2,X_1],X_1]&=-2X_1(X_1+X_2-I). \label{L2X1X1t}
\end{align}
\end{subequations}
These commutation relations can be obtained from those, given in \eqref{L3X1L3tt}, of the centralizer $\mathfrak{C}_{X_3}(\mathfrak{J}_2)$ under the substitution
\begin{equation}
    L_3 \rightarrow L_2, \quad X_3 \rightarrow X_2; \qquad b \rightarrow c.
\end{equation}
Given the analysis of Subsection \ref{subs:X3}, upon performing these replacements in \eqref{JacsubL3X1} and \eqref{parCX3} we conclude that $L_2$ and $X_1$ realize the $\mathfrak{J}_1$ relations \eqref{rank1} with
\begin{equation}
    K_1=L_2, \quad K_2=1-\frac{X_1}{1-X_2}; \qquad \alpha = c, \quad \beta=a.
\end{equation}
It then follows from \eqref{overJac} that
\begin{equation}
    {}^{\scriptstyle \sigma}\langle y', k';a,c \mid x, y\rangle = \delta (y-y')\; E_{k'}(x,y) \;J_{k'}^{(c,a)}\left(1-\frac{x}{1-y}\right),
\end{equation}
with $E_k(x,y)$ such that ${}^{\scriptstyle \sigma}\langle y', k';a,c \;|\;y'', k'';a,c\rangle ^{\sigma} = \delta _{k',k''} \delta (y'-y'')$. Using the Kummer transformation \eqref{1-x} and repeating computations that we have done before to determine $E_k(x,y)$ we arrive at 
\begin{equation}
        {}^{\scriptstyle \sigma}\langle y', k';a,c \mid x, y\rangle = \delta (y-y')\; \sqrt{\frac{1}{N_k'^{(a,c)}}\frac{x^a(1-x-y)^c}{(1-y)^{a+c+1}}} \;J_{k'}^{(a,c)}\left(\frac{x}{1-y}\right). \label{oveL1X2}
\end{equation}
(ii) We need next to determine the overlaps ${}^{\scriptstyle \sigma}\langle n,k;a,b,c \mid  y', k';a,c\rangle ^{\sigma}$ between the eigenstates of $L$ and $X_2$. To that end we must identify the algebra generated by these two elements that commute with $L_2$. From the defining relations of $\mathfrak{J}_2$  given in the Appendix \ref{APP A}, we find the following commutation relations between $L$ and $X_2$:
\begin{subequations}\label{LX2Ltt}
    \begin{align}
    [[L,X_2]\;,L\;]&=2\{ X_2,L\}-2L+2 L_2
    -(a+b+c+1)((a+b+c+3)X_2 - (b+1)I),\label{LX2Lt}\\
    [[L,X_2],X_2]&=-2X_2^2+2X_2.\label{LX2X2t}
\end{align}
\end{subequations}
These can be obtained from the relations \eqref{LX3Ltt}  for the centralizer $\mathfrak{C}_{L_3}(\mathfrak{J}_2)$ under the changes:
\begin{equation}
    X_3 \rightarrow X_2, \quad L_3 \rightarrow L_2; \qquad c \rightarrow b, \quad b \rightarrow c. \label{substCL3toCL2}
\end{equation}
From the study performed in Section \ref{subs:L3} of $\mathfrak{C}_{L_3}(\mathfrak{J}_2)$, we may thus infer that $L$ and $X_2$ also form a rank one Jacobi algebra. Applying the substitution \eqref{substCL3toCL2} to the equations \eqref{idenCL3} and \eqref{parCL3} carries the identification with the presentation \eqref{rank1} of $\mathfrak{J}_1$ which yields on the one hand, that the central $L_2$ is consistently given by $L_2=-k(k+a+c+1)$ on the relevant irreducible $\mathfrak{J}_1$-module and, on the other hand, that
\begin{equation}
    K_1=L+k(k+a+b+c), \quad K_2=X_2; \qquad \alpha=b, \quad \beta=2k+a+c+1.
\end{equation}
With these observations and imposing the orthonormality of the elements of the set $\{|\;n,k;a,b,c\rangle ^{\sigma}\}$, we may write that the overlaps between these basis vectors and $|y', k'; a,c\rangle^{\sigma}$ are
\begin{equation}
    {}^{\scriptstyle \sigma}\langle n,k;a,b,c \;|\;y',k';a,c\rangle ^{\sigma} = \delta_{k,k'} \; \sqrt{\frac{(y')^b (1-y')^{a+c+2k+1}}{N_{n-k}^{(b,a+c+2k+1)}}}\;J_{n-k}^{(b, a+c+2k+1)} (y'). \label{sigma-int}
\end{equation}
(iii) Inserting now, \eqref{oveL1X2} and \eqref{sigma-int} in \eqref{conv3} gives 
\begin{align}
       &{}^{\scriptstyle \sigma}\langle n,k;a,b,c \mid x,y\rangle = \nonumber \\
       &\sqrt{\frac{x^a y^b (1-x-y)^c}{N_k^{(a,c)}N_{n-k}^{(b,a+c+2k+1)}}}\;J_{n-k}^{(b, a+c+2k+1)} (y)\; (1-y)^k \;J_{k}^{(a,c)}\left(\frac{x}{1-y}\right), \label{sigmaJ}
\end{align}
which confirms Proposition \ref{pro:4} in light of the definition \eqref{defJac2} of the two-variable Jacobi polynomials.
\end{proof}

The overlaps of Proposition \ref{pro:4} are decomposed into those of Proposition \ref{prop 1} according to the formula 
\begin{equation}
   {}^{\scriptstyle \sigma}\langle n,\ell;a,b,c\;|x, y\;\rangle=
   \sum_{m=0}^{n} {}^{\scriptstyle \sigma}\langle n,\ell;a,b,c\;|\;n, m; a, b, c\;\rangle \langle \;n, m;a, b, c\;|\;x, y \rangle, \label{expansion} 
\end{equation}
which involves the overlaps ${}^{\scriptstyle \sigma}\langle n,\ell;a,b,c\;|\;n, m; a, b, c\;\rangle$ between the bases attached to the bottom of Figure \ref{fig:4} that respectively diagonalize $L_2$ and $L_1$, see \eqref{sigmab2} and \eqref{bl1} with eigenvalues $-\ell(\ell+a+c+1)$ and $-m(m+b+c+1)$.
As is familiar, these matrix elements can be obtained from the identification of the rank one algebra generated by $L_1$ and $L_2$ which is expected to be of Racah type. Indeed, it is confirmed with the help of the defining relations of $\mathfrak{J}_2$ given in Appendix \ref{APP A}, that $L_1$ and $L_2=L-L_1-L_3$ obey
\begin{subequations}\label{L1L2L1tt}
    \begin{align}
[[L_1,L_2],L_1]&=2\{L_1,L_2\}+2L_1^2-2L_1L 
    +(b+c)(c+1)(L-L_1-L_2) \nonumber\\&
    - (b+c)(c+1)L_2-(c-b)(a+1)L_1, \label{L1L2L1t}\\
[[L_1,L_2],L_2]&=-2\{L_1,L_2\}-2L_2^2+2L_2L  
    -(a+c)(c+1)(L-L_1-L_2) \nonumber\\&
    + (a+c)(a+1)L_1+(c-a)(b+1)L_2, \label{L1L2L2t}
 \end{align}
\end{subequations}
These commutation relations can be obtained from those of the centralizer $C_L(\mathfrak{J}_2)$ of $L$ given in equations \eqref{L1L3L1tt} under the substitution 
\begin{equation}
    L_3 \rightarrow L_2; \qquad c \rightarrow b, \quad b \rightarrow c,
\end{equation}
that has been pervasive in the analysis of this section.
We can thus avail ourselves of the identification of $C_L(\mathfrak{J}_2)$ performed in Subsection \ref{subs:L} to conclude that $L_1$ and $L_2$ form the algebra $\mathfrak{R}_1$ with the relations \eqref{L1L2L1tt} mapped onto \eqref{Ract} by setting $K_1=L_1$ and $K_2=L_2$ and the parameters $\xi, \;\eta_1,\; \eta_2,\; \zeta_1,\; \zeta_2$ obtained from \eqref{xiJac}, \eqref{eta1Jac}, \eqref{eta2Jac}, \eqref{zeta1Jac}, \eqref{zeta2Jac} upon the substitutions $c \rightarrow b$ and $b \rightarrow c$. Mindful of these parameter exchanges, the argumentation of Section \ref{sec:over} can then be repeated to find that
\begin{align}
     {}^{\scriptstyle \sigma}\langle n,\ell;a,b,c\;|\;n,m;a,b,c\rangle &= 
    (-1)^m  \sqrt{\frac{w^{(c,b,-n-1,n+1+a+c)}(\ell)}{M_m^{(c,b,-n-1,n+1+a+c)}}}R_m\left(\lambda^{(a,c)}(\ell); c,b,-n-1, n+1+a+c\right), \nonumber \\
  &=(-1)^m \;S_m^{(c, b, -n-1, n+1+a+c)}(\ell)
    ,\label{21over}
\end{align} 
and to make explicit the expansion \eqref{expansion}. The corollaries analogous to those presented in Section \ref{sec:over}, such as Proposition \ref{pro:3}, can be derived in a similar fashion in this case. We have thus shown how the algebraic interpretation of the two-variable Jacobi polynomials meshes with the triangle symmetries. We shall dwell more on these transformations next.
\subsection{Reflections and rotations}
Let us denote, in counterclockwise order, by $a, b, c$,  the three vertices of an equilateral triangle. In addition to the identity $e$, the symmetry transformations of this figure are:
\begin{itemize}
    \item The reflection $\pi$: $a \leftrightarrow c, b \leftrightarrow b$;
    \item The reflection $\sigma$: $a \leftrightarrow b, c \leftrightarrow c$;
    \item The reflection $\tau$: $a \leftrightarrow a, b \leftrightarrow c$;
    \item The rotation by $\frac{2\pi}{3} , R(\frac{2\pi}{3})$: $a \rightarrow b, b \rightarrow c, c \rightarrow a$;
    \item The rotation by $\frac{4\pi}{3} , R(\frac{4\pi}{3})$: $a \rightarrow c, b \rightarrow a, c \rightarrow b$.
\end{itemize}

\noindent It is well known that these operations form the dihedral group $D_3$ which is isomorphic to the symmetric group $S_3$. The Cayley multiplication table is:

\smallskip

\[
\begin{array}{c|cccccc}
      & e & \pi & \sigma & \tau & R\!\left(\tfrac{2\pi}{3}\right) & R\!\left(\tfrac{4\pi}{3}\right) \\
\hline
e     & e & \pi & \sigma & \tau & R\!\left(\tfrac{2\pi}{3}\right) & R\!\left(\tfrac{4\pi}{3}\right) \\
\pi   & \pi & e & R\!\left(\tfrac{2\pi}{3}\right) & R\!\left(\tfrac{4\pi}{3}\right) & \sigma & \tau \\
\sigma& \sigma & R\!\left(\tfrac{4\pi}{3}\right) & e & R\!\left(\tfrac{2\pi}{3}\right) & \tau & \pi \\
\tau  & \tau & R\!\left(\tfrac{2\pi}{3}\right) & R\!\left(\tfrac{4\pi}{3}\right) & e & \pi & \sigma \\
R\!\left(\tfrac{2\pi}{3}\right) & R\!\left(\tfrac{2\pi}{3}\right) & \tau & \pi & \sigma & R\!\left(\tfrac{4\pi}{3}\right) & e \\
R\!\left(\tfrac{4\pi}{3}\right) & R\!\left(\tfrac{4\pi}{3}\right) & \sigma & \tau & \pi & e & R\!\left(\tfrac{2\pi}{3}\right) \\
\end{array}
\]


\subsubsection{Action of $D_3$ on the two-variable Jacobi polynomials $J_{n,k}^{(a,b,c)}(x,y)$} In determining the actions on these polynomials of the reflections and rotations given above, we shall keep in mind their definition given in \eqref{defJac2} and remember that the transformations of the variables $x, y, 1-x-y$ follow those of the parameters $a, b, c$ through the pairings $(a,x), (b,y), (c, 1-x-y)$ as was done before.

Let us first observe that the polynomials $J_{n,k}^{(a,b,c)}(x,y)$ are only affected by a phase under the transposition $\tau$:
\begin{lemma}\label{lem:61} The polynomials $J_{n,k}^{(a,b,c)}(x,y)$ vary by a phase under the reflection that exchanges the indices $2$ and $3$:
    \begin{align}
    \tau \circ J_{n,k}^{(a,b,c)}(x,y) &= J_{n,k}^{(a,c,b)}(x,1-x-y) \nonumber \\
    &= (-1)^k J_{n,k}^{(a,b,c)}(x,y). \label{tauJ}
    \end{align}
\end{lemma}
\begin{proof}
    From the definition \eqref{defJac2} we have that:
    \begin{align}
        J_{n,k}^{(a,c,b)}(x,1-x-y)&=J_{n-k}^{(a, b+c+2k+1)}(x) \; (1-x)^k \; J_k^{(c,b)}\left(\frac{1-x-y}{1-x}\right) \nonumber \\
        &= J_{n-k}^{(a, b+c+2k+1)}(x)\;(1-x)^k \;J_k^{(c,b)}\left(1-\frac{y}{1-x}\right),
    \end{align}
   which implies \eqref{tauJ} using \eqref{1-x}.
\end{proof}
We have introduced up to now the two additional families
\begin{align}
\pi \circ J_{n,k}^{(a,b,c)}(x,y)&=J_{n,k}^{(c,b,a)}(1-x-y,y); \\
\sigma \circ J_{n,k}^{(a,b,c)}(x,y)&=J_{n,k}^{(b,a,c)}(y,x),
    \end{align}
that appear in particular in Proposition \ref{pro:2} and Proposition \ref{pro:4}. Let us now observe that the actions of the rotations $R(\frac{2 \pi}{3})$ and $R(\frac{4 \pi}{3})$ do not introduce new sets of two-variable Jacobi polynomials. 

\begin{lemma} The actions of the rotations $R(\frac{2 \pi}{3})$ and $R(\frac{4 \pi}{3})$ on 
 $J_{n,k}^{(a,b,c)}(x,y)$ yield up to a sign the polynomial bases that are obtained from the action of $\sigma$ and $\pi$ respectively:
 \begin{align}
     &R(\frac{2\pi}{3}) \circ J_{n,k}^{(a,b,c)}(x,y)=(-1)^k \; \sigma \circ J_{n,k}^{(a,b,c)}(x,y); \label{2pi}\\
     &R(\frac{4\pi}{3}) \circ J_{n,k}^{(a,b,c)}(x,y)=(-1)^k \; \pi \circ J_{n,k}^{(a,b,c)}(x,y). \label{4pi}
 \end{align}
\end{lemma}
\begin{proof}
    Consider the action of $R(\frac{2\pi}{3})$:
    \begin{align}
        R(\frac{2\pi}{3}) \circ J_{n,k}^{(a,b,c)}(x,y)&=J_{n,k}^{(b,c, a)}(y,1-x-y)\nonumber \\
        &=J_{n-k}^{(b,a+c+2k+1)}(y)\;(1-y)^k \;J_k^{(c,a)} \left(\frac{1-x-y}{1-y} \right)\nonumber\\
        &=(-1)^k \;J_{n-k}^{(b,a+c+2k+1)}(y)\;(1-y)^k \;J_k^{(a,c)} (\frac{x}{1-y}),
    \end{align}
    which confirms \eqref{2pi} in view of \eqref{sigmaJ}. The proof of \eqref{4pi} proceeds in a similar fashion.
\end{proof}

\subsection{Revisiting the expansions between the polynomials $J_{n,k}^{(a,b,c)}(x,y)$, $ \pi \circ J_{n,k}^{(a,b,c)}(x,y)$, and $\sigma \circ J_{n,k}^{(a,b,c)}(x,y)$}

Let us return now to the change of basis matrices that connect the three bases of two-variable Jacobi polynomials obtained through the action of the triangle reflections. We shall determine how these can be derived from the one given in Proposition \ref{pro:4} using the dihedral transformations. In terms of the orthonormalized functions $S_m(\ell)$ defined in \eqref{orthfn}, one can rewrite as follows the result of Proposition \ref{pro:4} that provides the representation of the reflection $\pi$ in the basis $J_{n,k}^{(a,b,c)}(x,y)$:
\begin{align}
    &\frac{1}{\sqrt{N_{\ell}^{(b,a)}N_{n-\ell}^{(c,a+b+2\ell+1)}}}\; J_{n,\ell}^{(c,b,a)}\left[(1-x-y),y\right] =\nonumber \\
    &\sum_{m=0}^n\;(-1)^m \; S_m ^{(b,c,-n-1,n+1+a+b)}(\ell) \  \frac{1}{\sqrt{N_m^{(b,c)}N_{n-m}^{(a, b+c+2m+1)}}}\; J_{n,m}^{(a,b,c)}(x,y). \label{piexp}
\end{align}
We considered at the end of Subsection \ref{sec:bac} the relation between the polynomials $\sigma \circ J_{n,k}^{(a,b,c)}(x,y)$ and the original $J_{n,k}^{(a,b,c)}(x,y)$. It is obtained from \eqref{eqprop4}, \eqref{expansion} and \eqref{21over} and expressed in terms of the orthonormalized functions, it reads:
\begin{align}
    &\sigma \circ  \left[\frac{1}{\sqrt{N_{\ell}^{(b,c)}N_{n-\ell}^{(a, b+c+2\ell+1)}}}\; J_{n,\ell}^{(a,b,c)}(x,y)\right] =
    \frac{1}{\sqrt{N_{\ell}^{(a,c)}N_{n-\ell}^{(b,a+c+2\ell+1)}}}\; J_{n,\ell}^{(b,a, c)}(y, x) \nonumber \\
    &=\;\sum_{m=0}^n\;(-1)^m \; S_m ^{(c,b,-n-1,n+1+a+b)}(\ell) \  \frac{1}{\sqrt{N_m^{(b,c)}N_{n-m}^{(a, b+c+2m+1)}}}\; J_{n,m}^{(a,b,c)}(x,y). \label{sigmaexp}
\end{align}
As an outcome of the discussion of the triangle symmetries, it is seen that the above relation can also be derived from the expansion formula given in Proposition \ref{pro:4}.
\begin{cor}\label{cor:1}
    The equation \eqref{sigmaexp} that provides the representation of the reflection $\sigma$ in the basis ${J_{n,k}^{(a,b,c)}(x,y)}$ can naturally be obtained through symmetry methods, from the similar representation of the reflection $\pi$ given in Proposition \ref{pro:4} and reproduced in equation \eqref{piexp}.
\end{cor}
\begin{proof}
    With the help of \eqref{4pi}, we have:
    \begin{equation}
          \pi \circ  \left[\frac{1}{\sqrt{N_{\ell}^{(b,c)}N_{n-\ell}^{(a, b+c+2\ell+1)}}}\; J_{n,\ell}^{(a,b,c)}(x,y)\right] = 
          \frac{(-1)^{\ell}}{\sqrt{N_{\ell}^{(b,a)}N_{n-\ell}^{(c, a+b+2\ell+1)}}}\; R(\frac{4\pi}{3}) \circ  J_{n,\ell}^{(a,b,c)}(x,y).
    \end{equation}
    Equating this last expression to the right-hand side of \eqref{piexp} and applying $R(\frac{2\pi}{3}) $ to both sides of the resulting equation (remembering that $R(\frac{2\pi}{3}): a \rightarrow b, b \rightarrow c, c \rightarrow a$ and that $R(\frac{2\pi}{3})\circ R(\frac{4\pi}{3})=e$), one finds
    \begin{align}
                  &\frac{(-1)^{\ell}}{\sqrt{N_{\ell}^{(c,b)}N_{n-\ell}^{(a, b+c+2\ell+1)}}}\;   J_{n,\ell}^{(a,b,c)}(x,y) \nonumber \\
                  &= \sum_{m=0}^n\;(-1)^m \; S_m ^{(c,a,-n-1,n+1+b+c)}(\ell) \  \frac{1}{\sqrt{N_m^{(c,a)}N_{n-m}^{(b, a+c+2m+1)}}}\;R(\frac{2\pi}{3}) \circ J_{n,m}^{(a,b,c)}(x,y).
    \end{align}
    Calling upon the action of $R(\frac{2\pi}{3}) $ given in \eqref{2pi} and the orthonormality of the functions $S_m(\ell)$ as per \eqref{orthS}, we find
    \begin{align}
    &\frac{1}{\sqrt{N_k^{(c,a)}N_{n-k}^{(b,a+c+2k+1)}}}\;\sigma \circ J_{n,k}^{(a,b,c)}(x, y) =\nonumber \\
    &\sum_{\ell=0}^n\;(-1)^{\ell} \; S_k ^{(c,a,-n-1,n+1+b+c)}(\ell) \  \frac{1}{\sqrt{N_{\ell}^{(c,b)}N_{n-\ell}^{(a, b+c+2\ell+1)}}}\; J_{n,\ell}^{(a,b,c)}(x,y). \label{sigmaexpsym}
\end{align}
With the help of the duality property \eqref{duality} of the orthonormalized functions which gives here
\begin{equation}
    S_k ^{(c,a,-n-1,n+1+b+c)}(\ell)=S_{\ell} ^{(c,b,-n-1,n+1+a+c)}(k),
\end{equation}
we see that \eqref{sigmaexpsym} becomes \eqref{sigmaexp} upon the appropriate renaming of the indices thereby proving the corollary.
\end{proof}

\begin{remark} A sign ambiguity may arise through the normalization process of the states and becomes of significance in relative comparisons of the bases. The proof of Corollary \ref{cor:1}, justifies the presence of the $(-1)^m$ factor in the expansion of vectors of one basis into those of another. Indeed, in its absence, the formulas such as \eqref{piexp} or \eqref{sigmaexp}, for the change of basis would be asymmetrical and not respect the covariance of the formalism. The presence of this factor can also be checked explicitly.
\end{remark}

To complete the picture one might wish to have the expansion of the polynomials $\sigma \circ J_{n,k}^{(a,b,c)}(x, y) =  J_{n,k}^{(b,a,c)}(y, x)  $ in terms of the elements of the set $\{\pi \circ J_{n,k}^{(a,b,c)}(x, y)=J_{n,k}^{(c,b,a)}\left[(1-x-y), y\right] , k=0,\dots,n \}$. Here it is:

\begin{cor}\label{cor:2}
    \begin{align}
    &\frac{1}{\sqrt{N_{\ell}^{(a,c)}N_{n-\ell}^{(b,a+c+2\ell+1)}}}\; J_{n,\ell}^{(b,a,c)}(y,x) =\nonumber \\
    &\sum_{m=0}^n\; S_m ^{(a,b,-n-1,n+1+a+c)}(\ell) \  \frac{1}{\sqrt{N_m^{(a,b)}N_{n-m}^{(c, a+b+2m+1)}}}\; J_{n,m}^{(c,b,a)}\left[(1-x-y),y\right]. \label{sigmapiexp}
\end{align}
\end{cor}

\begin{proof}
    Apply $R(\frac{4\pi}{3})$ to both sides of the relation \eqref{piexp}. Use the fact that $R(\frac{4\pi}{3})\circ \pi = \sigma$ and formula \eqref{4pi} to arrive at \eqref{sigmapiexp}.
\end{proof}

\subsection{The overall picture}\label{sec:overall}

The general covariance of the algebraic interpretation scheme for the two-variable Jacobi polynomials is captured in Figure \ref{fig:6} below. Starting, with the initial pentagon of Figure \ref{fig:tr}, two additional pentagons generated by successive rotations of $\frac{2\pi}{3}$ are displayed around a triangle formed by their bases. The pentagon of Figure \ref{fig:4} is the one obtained after the first rotation of $\frac{2\pi}{3}$. The generators at the summit of each pentagon arise from cyclical permutation so as to have the sequence $(X_1, X_3), (X_2, X_1), (X_3, X_2)$ reading in anticlockwise fashion. Note that the labels attached to the vertices of each pentagon are mapped onto one another by the reflections about the axis that passes through their summit and crosses their base at $\frac{\pi}{2}$. These reflections are respectively, $\pi, \sigma, \tau$. The vertices connected by segments formed of dashes and dots should be identified and have been blown up to exhibit the symmetry of the whole diagram with respect to reflections about these axes and to allow for the appropriate positioning of the generators according to which pentagon a vertex is taken to belong.

To each vertex is associated a representation basis of the rank two Jacobi algebra $\mathfrak{J}_2$ defined by the joint eigenvectors of the generators attached to the vertex. Two-variable Jacobi polynomials arise as overlaps between bases associated to top and bottom vertices of a given pentagon. The pair of these polynomials intervening in one pentagon are mutually related by the reflection symmetry of that pentagon and the matrix connecting them are given in terms of Racah polynomials owing to the fact that two different representation bases of a centrally extended rank one Racah subalgebra of $\mathfrak{J}_2$ are associated to the bottom vertices of a pentagon.

Let us now explain how Figure \ref{fig:6} depicts the transformations of the two-variable Jacobi polynomials that we have discussed in this section. We shall start from the top of the diagram and look at it in a counterclockwise way. We recall that the standard two-variable Jacobi polynomials $J_{n,k}^{(a,b,c)}(x,y)$ are attached to the left side of the top pentagon. 

Note now that the bases corresponding to the simultaneous diagonalization of $X_1, X_3$ or $X_1, X_2$ or $X_2, X_3$ are all the same, i.e., $\{|\;x,y\;\rangle\}$ since
\begin{equation}
    X_1\;|\;x,y\;\rangle=x\;|\;x,y\;\rangle, \quad X_2\;|\;x,y\;\rangle=y\;|\;x,y\;\rangle, \quad X_3\;|\;x,y\;\rangle=\;(1-x-y)\;|\;x,y\;\rangle.
\end{equation}
It follows as a consequence that the polynomials corresponding to the right side of the second pentagon cannot be different from those ($J_{n,k}^{(a,b,c)}(x,y)$) associated of the left top-to-bottom path of the first pentagon since both subgraphs connect the same bases. The diagram shows that these two paths are images one of the other under the reflection $\tau$ about the axis passing through the summit $(X_3,X_2)$ of the third pentagon and the $(L,L_1)$ corner of the triangle. The requirement that these two branches lead to the same polynomials has been validated first, explicitly, in Subsection \ref{sec:bac} and confirmed, in Lemma \ref{lem:61}, from the symmetry perspective that is incorporated in Figure \ref{fig:6}. The same reasoning and result apply to establish that the polynomials attached to the left and right sides respectively of the third and first pentagons are the same up to signs as these arms of the corresponding pentagons are related by the reflection $\sigma$.

\begin{center}
\begin{tikzpicture}[scale=1.7]
   
    \def\R{2cm}\def\RT{0.8cm}
    \def\Rt{2.1cm}
    \def\Rtt{1.8cm}

    \foreach \i in {1,...,9} {
        \coordinate (P\i) at ({90 + (\i-1)*360/9}:\R);
    }

     \foreach \i in {1,...,3} {
        \coordinate (T\i) at ({-90 + (\i-1)*360/3}:\RT);
    }
    \draw[dashed] (T1) -- (T2) --(T3)--(T1);

\draw (T2)--(P8);\draw (T2)--(P9);
\draw (T3)--(P2);\draw (T3)--(P3);
\draw (T1)--(P5);\draw (T1)--(P6);

   \draw (P9) -- (P1);\draw (P1) -- (P2);
\draw[dash dot dot] (P2)--(P3) ;
\draw (P3) -- (P4);\draw (P4) -- (P5);
\draw[dash dot dot] (P5)--(P6) ;
\draw (P6) -- (P7);\draw (P7) -- (P8);

\draw[dash dot dot](P8)--(P9) ;

\node at ({-90 + (0)*360/3}:1) {$L$};
\node at ({-90 + (1)*360/3}:1) {$L$};
\node at ({-90 + (2)*360/3}:1) {$L$};

\node at ({-90 + (0)*360/3}:0.5) {$L_2$};
\node at ({-90 + (1)*360/3}:0.5) {$L_3$};
\node at ({-90 + (2)*360/3}:0.5) {$L_1$};

\node at ({90 + (+0.2)*360/9}:\Rt) {$X_3$};
\node at ({90 + (-0.2)*360/9}:\Rt) {$X_1$};
\node at ({90 + (-0.2+1)*360/9}:\Rt) {$L_1$};
\node at ({90 + (+0.15+1)*360/9}:\Rtt) {$X_1$};
\node at ({90 + (-0.15-1)*360/9}:\Rtt) {$X_3$};
\node at ({90 + (+0.2-1)*360/9}:\Rt) {$L_3$};
\node at ({90 + (-0.2-2)*360/9}:\Rt) {$L_3$};
\node at ({90 + (+0.15-2)*360/9}:\Rtt) {$X_3$};
\node at ({90 + (-0.15+2)*360/9}:\Rtt) {$X_1$};
\node at ({90 + (+0.2+2)*360/9}:\Rt) {$L_1$};
\node at ({90 + (-0.2+3)*360/9}:\Rt) {$X_2$};
\node at ({90 + (+0.2+3)*360/9}:\Rt) {$X_1$};
\node at ({90 + (-0.2+4)*360/9}:\Rt) {$L_2$};
\node at ({90 + (+0.15+4)*360/9}:\Rtt) {$X_2$};
\node at ({90 + (-0.15+5)*360/9}:\Rtt) {$X_2$};
\node at ({90 + (+0.2+5)*360/9}:\Rt) {$L_2$};
\node at ({90 + (-0.2+6)*360/9}:\Rt) {$X_3$};
\node at ({90 + (+0.2+6)*360/9}:\Rt) {$X_2$};
     \foreach \i in {1,...,9} {
        \draw[fill]  (P\i) circle (0.08); ;
     }

       \foreach \i in {1,...,3} {
        \draw[fill]  (T\i) circle (0.08); ;
     }

\end{tikzpicture}

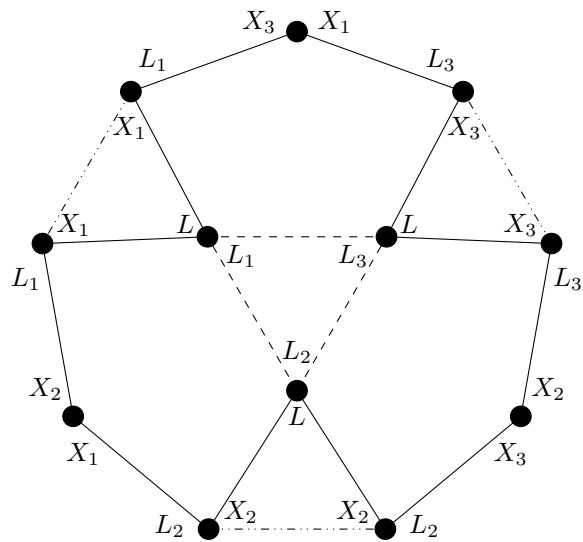
\captionof{figure}{This diagram depicts the symmetry of order three of the algebraic description of the two-variable Jacobi polynomials. Its detailed explanation is given in Subsection 6.4. \label{fig:6}}
\end{center}

\bigskip

The two other families of two-variable Jacobi polynomials that were introduced and characterized namely, $\pi \circ J_{n,k}^{(a,b,c)}(x,y)$ and $\sigma \circ J_{n,k}^{(a,b,c)}(x,y)$ were obtained and defined by applying the reflection symmetry $\pi$ of the first pentagon to the functions attached to its left side and similarly for the other set, by applying the reflection symmetry $\sigma$ to the polynomials attached to the right side of the second pentagon which are essentially, as we just discussed, those attached to the left side of the first (top) pentagon.

Owing to the properties \eqref{2pi} and \eqref{4pi} according to which the actions of $R(\frac{2\pi}{3})$ and $R(\frac{4\pi}{3})$ on the polynomials $J_{n,k}^{(a,b,c)}(x,y)$ equate, up to signs, to the actions respectively of $\sigma$ and $\pi$ on the same polynomials, we see that this is reflected in Figure 6 in the fact that the left arm of the top pentagon goes into the left side of pentagon 2 and then onto the left top-to-bottom path of pentagon 3 (which is known to be "equivalent" to the right one of pentagon 1) under counterclockwise rotations of $\frac{2\pi}{3}$ and $\frac{4\pi}{3}$.

Finally, given that the Racah algebra presentation is also covariant under the dihedral transformations, Figure \ref{fig:6} naturally reflects that the expansions of $\sigma \circ J_{n,k}^{(a,b,c)}(x,y)$ in terms of $J_{n,k}^{(a,b,c)}(x,y)$ (Corollary \ref{cor:1}) and of $\pi \circ J_{n,k}^{(a,b,c)}(x,y)$  in terms of $\sigma \circ J_{n,k}^{(a,b,c)}(x,y)$ (Corollary \ref{cor:2}) can be obtained by rotations of $R(\frac{2\pi}{3})$ and $R(\frac{4\pi}{3})$ of the relation of Proposition \ref{pro:4} (and recast in \eqref{piexp}) giving $\pi \circ J_{n,k}^{(a,b,c)}(x,y)$ as an expansion over the basis $ J_{n,k}^{(a,b,c)}(x,y)$ with univariate Racah polynomials as coefficients.

\section{Conclusion \label{sec:conc}}

As mentioned in the introduction, the two-variable Jacobi polynomials on the triangle have been studied and used in various contexts over many decades. The relation between special functions and representation theory enjoys a secular tradition and it is in this fruitful spirit that we aimed to provide an algebraic interpretation of these bivariate orthogonal polynomials. This could be achieved by introducing the Jacobi algebra of rank two denoted by $\mathfrak{J}_2$ and showing that the two-variable Jacobi polynomials arise as overlaps between representation bases of this algebra.

Let us now recap the main elements of this characterization. The algebra $\mathfrak{J}_2$ is presented in terms of generators and relations in Appendix \ref{APP A}. It has five generators: $X_1, X_3, L_1, L_3, L$ and five maximal abelian subalgebras: $(X_1,X_3), (L_,L_1), (L,L_3), (X_1, L_1), (X_3, L_3)$.
We consider a representation in which all these generators are symmetrizable. To each of the maximal abelian algebras (MAA), there is associated a representation basis defined by the joint eigenvectors of the generators of the particular MAA. We shall designate below these bases by the pair of generators they diagonalize. The overlaps between these basis vectors will produce orthogonal functions that are by construction bispectral.

The characterization of the two-variable Jacobi polynomials $J_{n,k}^{(a,b,c)}(x,y)$ exploits the subalgebra structure of $\mathfrak{J}_2$ and the representation theory of the rank one Jacobi and Racah algebras. We focused on the overlaps between the bases $(X_1, X_3)$ and $(L,L_1)$ and between the bases $(X_1, X_3)$ and $(L,L_3)$ thus involving all five generators. We saw that this gave rise to two families of bivariate Jacobi polynomials. In the spirit of Leonard pairs \cite{terwilliger2003introduction}, it is of importance that the set of bases can be organized in two subsets as depicted by Figure \ref{fig:2} and Figure \ref{fig:3}, which begin with $(X_1, X_3)$ and end with either $(L,L_1)$ or $(L,L_3)$ and where successive couples of generators have a common element \cite{crampe2025factorized}. This has allowed to use the representation of $\mathfrak{J}_2$ to derive the explicit structure of the two-variable Jacobi polynomials. The change of basis matrix that relates the joint eigenvectors of $X_1$ and $X_3$ and those of $L$ and $L_1$ or $L_3$ can be obtained by convoluting the overlaps between these bases and the intermediate ones calling upon the resolution of the identity that these last bases provide. The generators that are different in successive basis-defining pairs form, as it turns out, rank one Jacobi algebras where the element that is common to the two pairs is central. The explicit expression of two-variable Jacobi polynomials as entangled products of standard one-variable Jacobi polynomials was seen to readily follow at that stage from the representation theory (that we have reviewed) of the rank one Jacobi algebra. This was one key result we wished to obtain and it is found in Propositions \ref{prop 1} and \ref{pro:2} that pertain to the two families resulting from the overlaps between the vectors of the basis $(X_1, X_3)$ with those of the basis $(L,L_1)$ on the one hand and $(L,L_3)$ on the other.

At this point, it is very natural to examine the overlaps between the end bases 
$(L,L_1)$ and $(L, L_3)$. This allows to loop the loop by connecting the two strands with three pairs of generators and to relate the two-variable Jacobi polynomials of one family to those of the other. As emphasized, the framework then lends itself naturally to a description using  pentagons as in Figure \ref{fig:tr}. This illustrates well that the polynomials of one family are obtained from those of the other from a reflection that we denoted by $\pi$. From the striking observation that the generators $L_1$ and $L_3$ realize a rank one Racah algebra centrally extended by $L$, we could obtain a formula expressing the reflected polynomials $\pi \circ J_{n,k}^{(a,b,c)}(x,y)$ in terms of the standard polynomials $J_{n,k}^{(a,b,c)}(x,y)$ with expansion coefficients given in terms of univariate Racah polynomials (Proposition \ref{pro:3}). This offered a new representation-theoretic proof of a result obtained by Dunkl \cite{dunkl1984orthogonal}. 

This relation that equates the convolution of three univariate polynomials (two of Jacobi type and one of Racah type) to another convolution of two univariate Jacobi polynomials is analogous to the pentagonal identity obeyed by the Racah polynomials and known as the Biedenharn-Elliott identity which is viewed as the archetype of pentagon relations. Pentagon maps appear in many areas notably in conformal and topological quantum field theory (see for instance \cite{moore1989classical}, \cite{reshetikhin1991invariants}, \cite{turaev2010quantum}). It might hence be of interest to reflect on the lights that the present concrete study brings in this respect.

The algebraic interpretation of the two-variable Jacobi polynomials on the triangle is expected to be covariant under the dihedral symmetry of the underlying simplex. This was seen to be so and the corresponding properties were discussed in details. These have been synthesized in Figure \ref{fig:6} which is thoroughly commented in Subsection \ref{sec:overall}.

It should be mentioned in concluding that a $q$ deformation of the two-variable Jacobi polynomials has been introduced in \cite{lewanowicz2010two}; these orthogonal polynomials are constructed in terms of two univariate big $q$-Jacobi polynomials. The $q=-1$ limit of these bivariate polynomials has been obtained in \cite{genest2015two}. These families of polynomials should lend themselves to interesting representation-theoretic treatments that would feature at their core the rank one Askey-Wilson \cite{zhedanov1991hidden}, \cite{crampe2021askey} and Bannai-Ito \cite{tsujimoto2012dunkl}, \cite{de2015bannai} algebras. Clearly, the extension of the present study to more than two variables would entail a richer structure and would certainly be also worth exploring. The embedding of $\mathfrak{J}_2$ in $\mathfrak{sl}_3$ \cite{rühl} and its role in connection with superintegrable models \cite{iliev2018symmetry}, \cite{genest2014superintegrability} are also research topics for the future.

\appendix
 \section{Structure relations of the rank two Jacobi algebra $\mathcal{J}_2$} \label{APP A}

 Here are the relations between the five generators $L, L_1, L_3, X_1, X_3$ that define the rank two Jacobi algebra $\mathcal{J}_2$.
 Recall that 
 \begin{equation}
     [L,L_1]=0, \; [L,L_3]=0, \; [L_1,X_1]=0, \; [L_3,X_3]=0, \; [X_1,X_3]=0. \label{zero}
 \end{equation}
 
\subsection{Commutators involving $[L,X_1]$}
 \begin{align}
     [[L,X_1]\;,L\;]&=2\{X_1,L\}-2L+2L_1-(a+b+c+1)((a+b+c+3)X_1 -(a+1)I),\label{LX1L}\\
     [[L,X_1],L_1]&=0 \quad \text{ by Jacobi identity}, \\
     [[L,X_1],L_3]&=\{X_1,L+L_3\}+\{X_3-I,L-L_1\}\nonumber\\
     & \quad -(a+b+c+1)\left((a+1)(X_1+X_3-I)+(b+1)X_1\right),\\ 
     [[L,X_1],X_1]&=-2X_1^2+2X_1,\label{LX1X1}\\
     [[L,X_1],X_3]&=-2X_1X_3.\label{LX1X3}
 \end{align}

\subsection{Commutators involving $[L,X_3]$}

\begin{align}
    [[L,X_3]\;,L\;]&=2\{ X_3,L\}-2L+2 L_3
    -(a+b+c+1)((a+b+c+3)X_3 - (c+1)I),\label{LX3L}\\
    [[L,X_3],L_1]&=\{X_1-I, L- L_3\}+\{X_3,L+L_1\}
\nonumber\\
& \qquad -(a+b+c+1)\left((c+1)(X_1+X_3-I)+(b+1)X_3\right),
\\
    [[L,X_3],L_3]&=0 \quad \text{by Jacobi identity},\\
    [[L,X_3],X_1]&=
     [[L,X_1],X_3] \quad \text{ by Jacobi identity}, \\
    [[L,X_3],X_3]&=-2X_3^2+2X_3.\label{LX3X3}
\end{align}

\subsection{Commutators involving $[L_1,L_3]$}
\begin{align}
[[L_1,L_3],\;L\;]&=0 \quad \text{by Jacobi identity},\\
[[L_1,L_3],L_1]&=2\{L_1,L_3\}+2L_1^2-2L_1L 
    +(b+c)(b+1)(L-L_1-L_3) \nonumber\\&
    - (b+c)(c+1)L_3-(b-c)(a+1)L_1, \label{L1L3L1}\\
[[L_1,L_3],L_3]&=-2\{L_1,L_3\}-2L_3^2+2L_3L  
    -(a+b)(b+1)(L-L_1-L_3) \nonumber\\&
    + (a+b)(a+1)L_1+(b-a)(c+1)L_3, \label{L1L3L3}\\
[[L_1,L_3],X_1]&=   -\{X_1-I,L-L_1-L_3\}-\{X_3,L-L_1\}
\nonumber\\
    &+(a+1)\left((b+1)X_3+(c+1)(X_1+X_3-I)\right), \\
[[L_1,L_3],X_3]&=\{X_1,L-L_3\} +\{X_3-I,L-L_1-L_3\}
\nonumber \\
    & 
-(c+1)\left((a+1)(X_1+X_3-I)+(b+1)X_1\right).
 \end{align}
 
\subsection{Commutators involving $[L_1,X_3]$}
\begin{align}
 [[L_1,X_3]\;,L\; ]&=
    [[L,X_3],L_1]\quad \text{ by Jacobi identity}, \\
 [[L_1,X_3],L_1]&=2\{X_3 ,L_1\}+\{X_1,L_1\}-2L_1  \nonumber \\
 &-(b+c) \left((b+c+2)X_3 -(c+1)(I-X_1) \right), \label{L1X3L1}\\
 [[L_1,X_3],L_3]&=
     [[L_1,L_3],X_3]\quad \text{ by Jacobi identity},\\
 [[L_1,X_3],X_1]&=0\quad \text{ by Jacobi identity},\\
 [[L_1,X_3],X_3]&=-2X_3^2+2(I-X_1)X_3.\label{L1X3X3}
\end{align}

\subsection{Commutators involving $[L_3,X_1]$}
\begin{align}
    [[L_3,X_1]\;,L\;]
     &=[[L,X_1],L_3]\quad \text{ by Jacobi identity},\\
    [[L_3,X_1],L_1]&=
   -[[L_1,L_3], X_1]\quad \text{ by Jacobi identity},\\
    [[L_3,X_1],L_3]&=\{2X_1+X_3-I,L_3\} \nonumber\\ 
 &-(a+b)\left((a+1)(X_1+X_3-I)+(b+1)X_1\right), \label{L3X1L3}\\
    [[L_3,X_1],X_1]&=-2X_1(X_1+X_3-I), \label{L3X1X1}\\
    [[L_3,X_1],X_3]&=0\quad \text{ by Jacobi identity}.
\end{align}
Note that the commutators down the list that can be obtained from the Jacobi identity $[[X,Y],Z]+[[Y,Z],X]+[[Z,X],Y] = 0$ are not repeated; rather, it is indicated to which expression previously given they are equal.

\subsection{Implied relations.}

Any commutation relations between two commutators can be computed from the previous relations. For example, let use the Jacobi identity and relations \eqref{LX1L}, \eqref{LX1X3} to transform the commutator $[[L,X_1],[L,X_3]]$:
\begin{align}
   &[[L,X_1],[L,X_3]]=[[[L,X_1],L],X_3]+[L,[[L,X_1],X_3]] \nonumber\\
   =&2[L_1,X_3]-2[L,X_3]+ 2[L,X_3] X_1-2[L,X_1]X_3.
\end{align}
Similar relations can be obtained for all the commutators between two commutators of the generators. These relations, together the defining ones, allows us to order the generators $L,L_1,L_3,X_1,X_3$ and the commutators $[L,X_1]$, $[L,X_3]$, $[L_1,L_3]$, $[L_1,X_3]$, $[L_3,X_1]$.

Let us mention that there exist additional relations between the ordered monomials. This feature has been already studied in \cite{crampe2021racah,post2024racah} for the higher rank Racah algebra. For example, the previous relation can be computed differently:
\begin{align}
   &[[L,X_1],[L,X_3]]=[[L,[L,X_3]],X_1]+[L,[X_1,[L,X_3]]] \nonumber\\
   =&2[L,X_1] -2[L_3,X_1]+2[L,X_1]X_3-2[L,X_3]X_1.
\end{align}
Comparison of the both results leads to 
\begin{align}
  [L_1,X_3]-[L,X_3]=[L,X_1] -[L_3,X_1]. \label{identity}
\end{align}

\section{Representation \label{sec:rep}}

We shall here provide the relevant representation of $\mathfrak{J}_2$ in terms of operators in the variables $x$ and $y$ that is, in the basis $\{|x,y\;\rangle |0\le x\le  1-y \le 1 \}$ and, in terms of difference operators in $n$ and $k$ that is in the basis $\{|n,k; a, b,c\rangle | n,k\in\mathbb{N}, 0\le k\le n\}$. It will be remarked that all generators are symmetrizable in these bases. Most results are taken from \cite{crampe2025two} where they are proved.

\subsection{Differential representation}

It is checked that the following operators in the variables $x$ and $y$ realize  the commutation relations given in Appendix \ref{APP A} (under the correspondence $W \rightarrow \widetilde{W}$, with $W=X_1, X_3, L_1, L_3, L$)
\begin{align}
        &\widetilde{X}_1=x,\\
    &\widetilde{X}_3=1-x-y,\\
    &\widetilde{L}_1=((b+1)(1-x)-(b+c+2)y)\partial_y + y(1-x-y) \partial_{yy},\label{L1}\\
    &\widetilde{L}_3=((a+1)y-(b+1)x)(\partial_x-\partial_y) + xy \left(\partial_{xx}+\partial_{yy}-2\partial_{xy}\right)\label{L3},\\
    &\widetilde{L}= x(1-x)\partial_{xx}+y(1-y)\partial_{yy} \nonumber \\
    &\quad-2xy\partial_{xy}+(a+1-(a+b+c+3)x)\partial_{x}+(b+1-(a+b+c+3)y)\partial_{y}.
\end{align}

\subsection{Difference realization}

Here again, it can be verified that the following difference operators acting on the degrees $n$ and $k$ obey the commutation relations that define $\mathfrak{J}_2$ (with $W \rightarrow \widehat{W}$).
\begin{align}
\widehat{X}_1&=-
\dfrac{(n-k+1)(n+k+a+b+c+2)}{(2n+a+b+c+2)(2n+a+b+c+3)} S_{+}\nonumber\\
&+\dfrac12\bigg(1-\dfrac{(2k+a+b+c+1)(2k-a+b+c+1)}{(2n+a+b+c+1)(2n+a+b+c+3)} \bigg)I\nonumber\\
&-\dfrac{(n-k+a)(n+k+b+c+1)}{(2n+a+b+c+1)(2n+a+b+c+2)}S_{-},\\
    \widehat{X}_3 &=
\dfrac{1}{2} I - \dfrac{1}{2} \widehat{X}_1    \nonumber\\&
+\dfrac{(k+1)(k+b+c+1)(n+k+a+b+c+2)(n+k+a+b+c+3) }{(2k+b+c+1)(2k+b+c+2)(2n+a+b+c+2)(2n+a+b+c+3)} 
S_{+}T_{+}\nonumber\\
&+\dfrac{2(k+1)(k+b+c+1)(n-k+a)(n+k+a+b+c+2) }{(2k+b+c+1)(2k+b+c+2)(2n+a+b+c+1)(2n+a+b+c+3)} 
T_{+}\nonumber\\
&+\dfrac{(k+1)(k+b+c+1)(n-k+a-1)(n-k+a) }{(2k+b+c+1)(2k+b+c+2)(2n+a+b+c+1)(2n+a+b+c+2)} 
S_{-}T_{+}\nonumber\\
&+\dfrac{(c^2-b^2)(n-k+1)(n+k+a+b+c+2) }{2(2k+b+c)(2k+b+c+2)(2n+a+b+c+2)(2n+a+b+c+3)}
S_{+}\nonumber\\
&+\frac{c^2-b^2}{4}\bigg(\dfrac{1}{(2k+b+c)(2k+b+c+2)}+\dfrac{1}{(2n+a+b+c+1)(2n+a+b+c+3)}
\nonumber\\
& \qquad  + \dfrac{1-a^2 }{(2k+b+c)(2k+b+c+2)(2n+a+b+c+1)(2n+a+b+c+3)} \bigg)  
I \nonumber\\
&+\dfrac{(c^2-b^2)(n-k+a)(n+k+b+c+1) }{2(2k+b+c)(2k+b+c+2)(2n+a+b+c+1)(2n+a+b+c+2)} 
S_{-}\nonumber\\
&+\dfrac{(k+b)(k+c)(n-k+1)(n-k+2) }{(2k+b+c)(2k+b+c+1)(2n+a+b+c+2)(2n+a+b+c+3)} 
S_{+}T_{-}\nonumber\\
&+\dfrac{2(k+b)(k+c)(n-k+1)(n+k+b+c+1) }{(2k+b+c)(2k+b+c+1)(2n+a+b+c+1)(2n+a+b+c+3)} 
T_{-}\nonumber\\
&+\dfrac{(k+b)(k+c)(n+k+b+c)(n+k+b+c+1) }{(2k+b+c)(2k+b+c+1)(2n+a+b+c+1)(2n+a+b+c+2)} 
S_{-}T_{-},\\
\widehat{L}_1&= -k(k+b+c+1)I,
\end{align}
\begin{align}
\widehat{L}_3&=
\frac{ \left(k+b  \right) \left(k+c  \right)\left(n-k+1\right) \left(n +k +b +c +1\right)}{\left(2k+b +c  \right) \left(2 k +b +c +1\right)}  
    T_{-}
    \nonumber \\&
    +\frac{\left(k +1\right) \left(k +b +c+1\right) \left(n -k+a \right) \left(n +k +a +b +c +2\right) }{\left(2 k +b +c +1\right) \left(2 k +b +c +2\right)}
    T_{+}
    \nonumber \\&
    +\biggl((k-n)(n-k+a+b+1) -\dfrac{k(k+c)(n-k+1) (n-k+a+1)}{2k+b+c} 
    \nonumber \\&\qquad
    + \dfrac{(k+1)(k+c+1)(n-k)(n-k+a)}{2k+b+c+2} \biggr) I ,\\
    \widehat{L}&=-n(n+a+b+c+2)I,
\end{align}
where the shifts $S_{\pm}$ and $T_{\pm}$ act as follows on generic vectors $|n,k\;\rangle$ labeled by these integers:
\begin{equation}
    S_{\pm}\;|n,k\;\rangle = |n\pm 1,k\;\rangle, \qquad    T_{\pm}\;|n,k\;\rangle = |n,k\pm 1\;\rangle.
\end{equation}

\subsection{Symmetrization and characterization of the polynomials }
We have found (see Proposition \ref{prop 1}) that 
 \begin{equation}
     \langle n,k;a,b,c\;|\;x,y\rangle =    \xi_{n,k}^{(a,b,c)} J_{n,k}^{(a,b,c)}(x,y)
 \end{equation}
 with \begin{equation}
     \xi_{n,k}^{(a,b,c)} = \sqrt{\frac{x^a y^b (1-x-y)^c}{N_k^{(b,c)}N_{n-k}^{(a,b+c+2k+1)}}}.
 \end{equation}
 It is observed that all the operators 
 \begin{equation}
      \xi_{n,k}^{(a,b,c)}\;\widetilde{W} \left[\xi_{n,k}^{(a,b,c)}\right]^{-1} \quad \text{and} \quad       \xi_{n,k}^{(a,b,c)}\;\widehat{W} \left[\xi_{n,k}^{(a,b,c)}\right]^{-1},
 \end{equation}
 are hermitian.
 
 \noindent It follows also that the relations
 \begin{equation}
     \widetilde{W}\; \left[ \xi_{n,k}^{(a,b,c)}\right]^{-1}\langle n,k;a,b,c\;|\;x,y\rangle = \widehat{W}\; \left[ \xi_{n,k}^{(a,b,c)}\right]^{-1}\langle n,k;a,b,c\;|\;x,y\rangle,
 \end{equation}
 for $W=L, L_1, X_1, X_3$ provide the recurrence relations and the differential equations of the two-variable Jacobi polynomials $J_{n,k}^{(a,b,c)}(x,y)$. These are given in more details in \cite{crampe2025two}.

\section*{Acknowledgments}
This work has been sponsored by a Québec-Kyoto cooperation grant from the Ministère des Relations Internationales et de la Francophonie of the Quebec Government. The authors acknowledge stimulating conversations with Jarmo Hietarinta, Sarah Post and Sasha Turbiner.
NC is partially supported by the international research project AAPT of the CNRS.
The research of ST is supported by JSPS KAKENHI (Grant Number 24K00528). LV is funded in part through a discovery grant of the Natural Sciences and Engineering Research Council (NSERC) of Canada. QL and LM enjoy postdoctoral fellowships provided by this fund.

\bibliographystyle{unsrt} 
\bibliography{ref_2varJac.bib}

\end{document}